\documentclass[10pt]{article}
\usepackage{amssymb}
\usepackage{amsmath}
\usepackage{amsthm}
\usepackage{amsbsy}
\usepackage{psfrag}
\usepackage{graphics}
\usepackage{graphicx}
\usepackage{subfigure}
\usepackage{epsfig}
\newtheorem{theorem}{Theorem}[section]
\newtheorem{lemma}[theorem]{Lemma}

\theoremstyle{definition}

\theoremstyle{remark}

\numberwithin{equation}{section}

\newcommand{\x}{\textbf{x}}
\newcommand{\y}{\textbf{y}}

\newcommand{\n}{\textbf{n}}
\usepackage{graphicx}
\usepackage{psfrag}
\usepackage{caption}
\usepackage{float}
\usepackage{sidecap}
\begin{document}
\title{\bf Resonance and Double Negative Behavior in Metamaterials\footnote{\normalsize Accepted for publication in Archive for Rational Mechanics and Analysis on
March 4, 2013}} 


\author{Yue Chen \\ {\normalsize Department of Mathematics} \\
{\normalsize University of Kentucky}\\
{\normalsize Lexington, KY 40506, USA.}\\ {\normalsize email: chenyue0715\symbol{'100}uky.edu}\\[3pt]
\and Robert Lipton \\{\normalsize Department of Mathematics}\\{\normalsize  Louisiana State University}\\
{\normalsize Baton Rouge, LA 70803, USA.}\\ {\normalsize email: lipton\symbol{'100}math.lsu.edu}\\[3pt]}




\maketitle







\begin{abstract}

A generic class of metamaterials is introduced and is shown to exhibit frequency dependent double negative effective properties. We develop a rigorous method for calculating the frequency intervals where either double negative or double positive effective properties appear and show how these intervals imply the existence of propagating Bloch waves inside sub-wavelength structures.  The  branches of the dispersion relation associated with Bloch modes are shown to be explicitly determined by the Dirichlet spectrum of the high dielectric phase and the generalized electrostatic spectra of the complement.
\end{abstract}

\section{Introduction}
\label{Introduction}
Metamaterials are new class of engineered materials that exhibit electromagnetic properties not readily found in nature. The novelty is that unconventional electromagnetic properties  can be created by carefully chosen sub-wavelength configurations of conventional  materials.   The distinctive properties of metamaterials are derived from geometrically induced resonances localized to specific frequencies. These resonances are used to control propagating modes with wavelengths longer than  the characteristic length scale of the material. Metamaterials  are envisaged for several application areas ranging from telecommunication and solar energy harvesting to the electromagnetic cloaking
of material objects. 

A generic metamaterial comes most often in the form of a crystal made from a periodic array of scatterers embedded within a host medium. The physical notions of frequency dependent effective magnetic permeability and dielectric permittivity are  used to describe the behavior of propagating modes at wavelengths larger than the length scale of the metamaterial crystal. The past decade has witnessed the development and identification of new sub-wavelength geometries for novel metamaterial properties. These include the simultaneous appearance of negative effective dielectric permittivity and magnetic permeability. Such ``left handed media'' are predicted to exhibit negative group velocity, inverse Doppler  effect, and an inverted Snell's law \cite{Veselago}.  The first metamaterial configurations  imparted electromagnetic properties consistent with the appearance of a negative bulk dielectric constant \cite{Pendry1998}. Subsequently electromagnetic behavior associated with negative effective magnetic permeability
at microwave frequencies were derived from periodic arrays of non-magnetic metallic split ring resonators \cite{PendryHolden}. 
Double negative or left handed metamaterials with simultaneous negative bulk permeability and permittivity at microwave frequencies have been verified for arrays of metallic posts and split ring resonators \cite{Smith}.  Subsequent work has delivered several new designs using 
different configurations of metallic resonators for double negative behavior \cite{23,2A,17,21,18,20,22}. 

Current state of the art metallic resonators do not perform well at optical frequencies  and 
alternate strategies are contemplated employing the use of both metals and dielectric materials for optical frequencies \cite{Shalaev}.
For higher frequencies in the infrared and optical range new strategies for generating double negative response rely on Mie resonances 
generated inside  coated rods consisting of a high  dielectric core coated with a dielectric exhibiting  plasmonic or Drude type frequency response at optical frequencies \cite{11,Yannopappas,Yanno2}. A second strategy for  generating double negative response employs dielectric resonances
associated with small rods or particles made from dielectric materials with large  permittivity, \cite{Plasmon,LPeng,VinkFelbacq}.  Alternate strategies for generating negative permeability at infrared and optical frequencies  use special configurations of plasmonic nanoparticles \cite{6A,7A,shevts}. The list of  metamaterial systems is rapidly growing  and comprehensive reviews of the subject can be found in \cite{Service} and \cite{Shalaev}.

Despite the large number of physically based strategies for generating unconventional properties the theory lacks mathematical frameworks that: 
\begin{enumerate}
\item Provide the explicit relationship connecting the leading order influence of  double negative  behavior to the existence of Bloch wave modes inside metamaterials.
\item Provide a systematic identification of the underlying spectral problems related to the crystal geometry that control the location of stop bands and propagation bands for metamaterial crystals.
\end{enumerate}

In this article we provide such a framework for a generic class of metamaterial crystals made from non-magnetic constituents. The crystal is given by a periodic array of two aligned non-magnetic rods; one of which possesses a large frequency independent dielectric constant while the other is characterized by a frequency dependent dielectric response.  In this treatment the frequency dependent dielectric response $\epsilon_P$ is associated with  plasmonic or Drude  behavior at optical frequencies given by \cite{11,Yannopappas}
\begin{eqnarray} 
\epsilon_P(\omega^2)=1-\frac{\omega_p^2}{\omega^2},
\label{singleosc}
\end{eqnarray} 
where $\omega$ is the frequency and $\omega_p$ is the plasma frequency \cite{bohren}. 
Here we develop a rigorous method for calculating the frequency intervals where either double negative or double positive bulk properties appear and show how these intervals imply the existence of Bloch wave modes in the dynamic regime away from the quasi static limit, see Theorems \ref{summable} and \ref{bands}. It is shown that 
these frequency intervals are explicitly determined by two distinct spectra. These are the Dirichlet spectrum of the Laplacian associated with the high dielectric rod and the electrostatic spectrum of a three phase high contrast medium obtained by sending the dielectric constant inside the high dielectric rod to $\infty$. The electrostatic spectra is introduced in Theorem \ref{completeeigen} and discussed in section \ref{genrealizedelectrostaticspectra}. The methods are illustrated  for  $\epsilon_P$ given by \eqref{singleosc}, however they apply to dielectrics characterized by single oscillator or multiple oscillator models that include dissipation and are of the form
\begin{eqnarray} 
\epsilon_P(\omega)=1+\sum_{j=1}^N\frac{\omega_{p}^2}{\omega^2_j-\omega^2-i\gamma_j\omega},
\label{multleosc}
\end{eqnarray} 
where $\omega_j$ are resonant frequencies, and $\gamma_j$ are  damping factors. 

We start with a metamaterial crystal characterized  by a period cell containing  two parallel infinitely long cylindrical rods. The rods are parallel to the $x_3$ axis and are periodically arranged within a square lattice over the transverse $\mathbf{x}=(x_1,x_2)$ plane. The period of the lattice is denoted by $d$. There is no constraint placed on the shape of the rod cross sections other than they have smooth boundaries and are simply connected.
The objective is to characterize the branches of the dispersion relation for for H-polarized Bloch-waves inside the crystal.   For this case the magnetic field is aligned with the rods and the electric field lies in the transverse plane. The direction of propagation is described by the unit vector $\hat{\kappa}=(\kappa_1,\kappa_2)$ and $k=2\pi/\lambda$ is the wave number for a wave of length~$\lambda$ and the fields are of the form
\begin{eqnarray}
H_3=H_3(\mathbf{x})e^{i(k\hat{\kappa}\cdot\mathbf{x}-t\omega )},\,\,E_1=E_1(\mathbf{x})e^{i(k\hat{\kappa}\cdot\mathbf{x}-t\omega)},\,\,
E_2=E_2(\mathbf{x})e^{i(k\hat{\kappa}\cdot\mathbf{x}-t\omega)} \label{em3}
\end{eqnarray}
where $H_3(\mathbf{x})$, $E_1(\mathbf{x})$, and $E_2(\mathbf{x})$ are  $d$-periodic for $\mathbf{x}$ in $\mathbb{R}^2$. 
In the sequel $c$ will denote the speed of light in free space. 
We denote the unit vector pointing along the $x_3$ direction by ${\bf e}_3$, and the periodic dielectric permittivity and magnetic permeability are denoted by $a_d$ and $\mu$ respectively. The electric field component $\mathbf{E}=(E_1,E_2)$ of the wave is determined by $${\bf E}=-\frac{ic}{\omega a_d}{\bf e}_3\times \nabla H_3.$$

The materials are assumed non-magnetic hence the magnetic permeability $\mu$ is set to unity inside the rods and host. 
The oscillating dielectric permittivity for the crystal 
is a $d$  periodic function in the transverse plane and
is described by $a_d=a_d(\mathbf{x}/d)$ where $a_d(\mathbf{y})$ is 
the unit periodic dielectric function taking the values
\begin{equation}
a_d(\y)=
\begin{cases}
\epsilon_H &\text{ in the host material},\\
\epsilon_P=\epsilon_P(\omega) & \text{ in the frequency dependent ``plasmonic'' rod},\\
\epsilon_R=\epsilon_r/d^2 &\text{ in the high dielectric rod} .
\end{cases}
\end{equation}
\par
This choice of high dielectric constant $\epsilon_R$ follows that of \cite{felbacqbouchette} where $\epsilon_r$ has dimensions of area. 
Setting $h^d(\mathbf{x})=H_3(\mathbf{x})e^{i(k\hat{\kappa}\cdot\mathbf{x})}$ the Maxwell equations take the form of the Helmholtz equation given by
\begin{equation}
-\nabla_{\mathbf{x}} \cdot \left(a_d^{-1}(\frac{\mathbf{x}}{d})\nabla_{\mathbf{x}} h^d(\mathbf{x})\right)=\frac{\omega^2}{c^2}h^d ~~~\text{ in  } \mathbb{R}^2.
\label{Helmholtzr2}
\end{equation}
The band structure is given by the Bloch eigenvalues $\frac{\omega^2}{c^2}$ which is a subset of the parameter space $\{\frac{\omega^2}{c^2} , -2\pi\leq k_1\leq 2\pi , -2\pi\leq k_2\leq 2\pi\}$, with $k=(k_1^2+k_2^2)^{1/2}$ and $\hat{\kappa}_i=k_i/k$. This constitutes the first Brillouin Zone for this problem.  

We set $\x=d\y$ for $\y$ inside the unit period $Y=[0,1]^2$, put  $\beta=dk\hat{\kappa}$ and write $u(\y)=H_3(d\y)$. The dependent variable is written $u^d(\y)=h^d(d\y)=u(\y)\exp^{i\beta\cdot\y}$,  and we recover the equivalent problem over the unit period cell given by
\begin{equation}
-\nabla_{\y} \cdot \left(a_d^{-1}(\y)\nabla_{\y} u^d\right)=\frac{d^2\omega^2}{c^2}u^d ~~~\text{ in  } Y.
\label{Helmholtzd}
\end{equation}

To proceed we work with the dimensionless ratio  $\rho=d/\sqrt{\epsilon_r}$, wave number $\tau=\sqrt{\epsilon_r}k$ and square frequency $\xi=\epsilon_r\frac{\omega^2}{c^2}$. The dimensionless parameter  measuring the departure away from the quasi static regime is given by the  ratio of period size to wavelength $\eta=dk=\rho\tau\geq 0$. The regime $\eta>0$ describes dynamic wave propagation while the infinite wavelength or quasi static limit is recovered for $\eta=0$. Metamaterials by definition are structured materials operating in the sub-wavelength regime $0<\eta<1$ away from the quasistatic limit \cite{PendryHolden}. For these parameters the dielectric permittivity takes the values  $\epsilon_P=1-\frac{\epsilon_r\omega_p^2/c^2}{\xi}$, $\epsilon_R=\frac{1}{\rho^2}$, $\epsilon_H=1$,  and  is denoted by $a_\rho(\y)$ for $\y$ in $Y$ and \eqref{Helmholtzd} is given by
\begin{equation}
-\nabla_{\y} \cdot \left(a_\rho^{-1}(\y)\nabla_{\y} u^d(\y)\right)=\rho^2\xi u^d(\y) ~~~\text{ in  } Y.
\label{Helmholtz}
\end{equation}
The unit period cell for the generic metamaterial system is represented in Fig. \ref{unitcell}. In what follows $R$ represents the rod cross section containing high dielectric material, $P$ the cross section containing the plasmonic material and $H$ denotes the connected host material region.

\par
\begin{figure}[h]
\centering
 \psfrag{n1}{$\n$}
\psfrag{H}{$H$}
\psfrag{P}{$P$}
\psfrag{R}{$R$}
\epsfig{figure=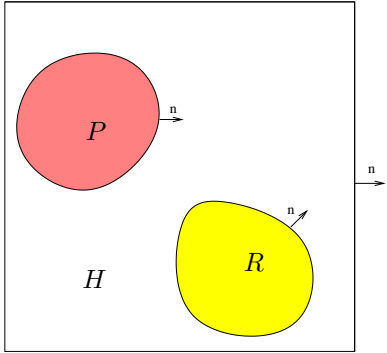 , width=2in}
\caption{Cross section of unit cell.}
\label{unitcell}
\end{figure}

It is shown in Theorem \ref{convergenceofxi} that the band structure for the metamaterial is characterized by a power series in $\eta$ and is governed by two distinct types of spectra determined by the shape and configuration of the rods inside the period cells. The series delivers the explicit relationship connecting the leading order influence of quasi static behavior, as mediated by {\em effective magnetic permeability and dielectric permittivity}, to the propagation of Bloch wave solutions inside metamaterial crystals made from sub-wavelength $1>\eta>0$ structures, see Theorems \ref{convergenceofxi}, \ref{summable}, and \ref{bands}.

The relevant spectra for this problem is found to be given by the Dirichlet spectra for the Laplacian over the rod cross sections $R$ together with a generalized electrostatic spectra associated with the infinite connected region exterior to the rods. These spectra provide two distinct criteria that taken together are sufficient for the existence of power series solutions see, Theorem \ref{summable}. In what follows the power series is developed in terms of a hierarchy of boundary value problems posed separately over the domain $R$ and the domain exterior to the high dielectric rod $Y\setminus R$. The existence of solutions for the boundary value problems inside the high dielectric rod $R$ is controlled by the Dirichlet spectra see, section \ref{higherorder}.  Existence of solutions for boundary value problems exterior to $R$ are determined by the generalized electrostatic spectrum see, Theorems \ref{completeeigen} and \ref{existencetheoremforoutsideR}. 
The electrostatic spectra is identified and is shown to be given by the eigenvalues of a compact operator acting on an appropriate Sobolev space of periodic functions see,  the discussion in section \ref{genrealizedelectrostaticspectra} and Theorem \ref{T}.  

The generic class of double negative metamaterials introduced here  appears to be new. The motivation behind their construction draws from earlier investigations. The influence of electrostatic resonances on the effective dielectric tensor associated with crystals made from a single frequency dependent dielectric inclusion is developed in the pioneering work of \cite{kantorbergman,McPhedranMilton,MiltonBook} see also the more recent work \cite{shevts} in the context of matematerials. These resonances are responsible for negative effective dielectric permittivity. In this context we point out that the sub-wavelength geometry introduced here is a three phase medium  and the categorization of all possible electrostatic resonances requires a different approach see section \ref{genrealizedelectrostaticspectra}.  On the other hand Dirichlet resonances  generate  negative effective permeability inside high contrast non-dispersive dielectric inclusions, this phenomena is discovered  in   \cite{felbacqbouchette,bouchettefelbacq,9A,3A,ObrienPendry}, see also the mathematically related investigations of \cite{Cherdansev,HempelandLenau,Smyshalaev,zhikov}. Motivated by these observations we have constructed a hybrid composite crystal that combines both high dielectric inclusions and frequency dependent inclusions for generating double negative response from non-magnetic materials. 

The power series approach to sub-wavelength $\eta<1$ analysis has been  developed in \cite{FLS} for characterizing the  dynamic dispersion relations for Bloch waves inside plasmonic crystals. It has also been applied to assess the influence of effective negative permeability on the propagation of Bloch waves inside high contrast dielectrics \cite{FLS2},  the generation of negative permeability inside metallic - dielectric resonators \cite{shipman}, and for concentric coated cylinder assemblages generating a double negative media \cite{chenlipton}.

We conclude noting that  earlier related work introduces the use of high contrast structures for opening band gaps in photonic crystals, this  is developed in \cite{FK1,FK2,FK3}. For two phase high contrast media integral equation methods are applied to recover  dispersion relations about frequencies corresponding to Dirichlet eigenvalues \cite{amari1} and \cite{amari2}.  The connection between high contrast interfaces and negative effective magnetic permeability for time harmonic waves is made in \cite{KohnShipman}.   More recently 
two-scale homogenization theory has been developed for three dimensional split ring structures that deliver negative effective magnetic permeability \cite{bouchetteschwizer,4A}. For periodic arrays made from metal fibers a homogenization theory delivering negative effective dielectric constant \cite{bouchettebourel} is established. A novel method for creating metamaterials with prescribed effective dielectric permittivity and effective magnetic permeability at a fixed frequency
is developed in  \cite{Milton2}.

\section{Background, basic theory, and generalized electrostatic resonances}
\label{Background}
In this section we outline the steps in the power series development of Bloch waves and establish the mathematical foundations for establishing existence of power series solutions.
We introduce the function space $H^1_{per}(Y)$ defined to be all $Y$-periodic, complex valued square integrable functions with square integrable derivatives with the usual inner product and norm given by 
\begin{eqnarray}
( u, v)_{Y}=\int_Y\left(\nabla u\cdot\nabla \overline{v}+u\overline{v}\right)\,dy \hbox{  and  }\Vert u\Vert_Y=(u,u)_Y^{1/2}.
\label{innerY}
\end{eqnarray}
We also introduce the space $H^1_{per}(Y\setminus R)$ given by all $Y$-periodic, complex valued square integrable functions with square integrable derivatives and the  inner product  and seminorm
\begin{eqnarray}
( u,v)=\int_{Y\setminus R}\nabla u\cdot\nabla \overline{v}\,dy  \hbox{  and  }\Vert u\Vert=(u,u)^{1/2}.
\label{innerproduct}
\end{eqnarray}
The variational form of (\ref{Helmholtz})  is given by 
\begin{equation}
\int_Y a_\rho^{-1}\nabla u^d\cdot \nabla \bar{\tilde{v}}= \int_Y \frac{\rho^2\xi}{c^2}u^d\bar{\tilde{v}}
\label{variational}
\end{equation}
for any $\tilde{v}=v(\y)e^{i\hat{\kappa}\cdot \tau\rho \y}$ , where $v \in H ^1_{per}(Y)$ . On writing $u^d=u(\y)e^{i\hat{\kappa}\cdot \tau\rho \y}$ and setting $\eta=\tau\rho$ we transform (\ref{variational}) into
\begin{eqnarray}
&&\int_H\tau^2(\xi-\epsilon_r\frac{\omega_p^2}{c^2})(\nabla+i\eta \hat{\kappa})u\cdot \overline{(\nabla+i\eta \hat{\kappa})v }+\int_P\tau^2\xi(\nabla+i\eta \hat{\kappa})u\cdot \overline{(\nabla+i\eta \hat{\kappa})v }\nonumber
\\&&+\int_R\eta^2(\xi-\epsilon_r\frac{\omega_p^2}{c^2})(\nabla+i\eta \hat{\kappa})u\cdot \overline{(\nabla+i\eta \hat{\kappa})v }=\int_Y\eta^2\xi(\xi-\epsilon_r\frac{\omega_p^2}{c^2})u\overline{v}
\label{variational2}
\end{eqnarray}
We introduce the power series
\begin{eqnarray}
 &&u=\sum_{m=0}^\infty \eta^m u_m \label{upower}
\\&&\xi=\sum_{m=0}^\infty\eta^m\xi_m \label{xipowerexpansion}
\end{eqnarray}
where $u_m$ belongs to $H^1_{per}(Y)$. In view of the algebra it is convenient to write $u_m=i^m\underline{u}_0\psi_m$ where $\underline{u}_0$ is an arbitrary constant factor. 

We now describe the underlying variational structure associated with the power series solution. Set $$z=\epsilon_P^{-1}(\xi_0)=\left(1-\frac{\epsilon_r\omega_p^2/c^2}{\xi_0}\right)^{-1}$$
and for $u$, $v$ belonging to $H^1_{per}(Y)$ we introduce the sesquilinear form
\begin{eqnarray}
B_z(u,v)=\int_{H}\nabla u\cdot\nabla\overline{v}\,d\y+\int_P\,z\,\nabla u\cdot\nabla\overline{v}\,d\y.
\label{bilinear}
\end{eqnarray}
Here $Y\setminus R=H\cup P$ and the form $B_z(u,v)$ is well defined for functions in $H^1_{per}(Y\setminus R)$.
Substitution of the series 
into the system (\ref{variational2}) and equating like powers of $\eta$ produce the infinite set of coupled equations for $m=0,1,2\ldots$ given by
\begin{eqnarray}
&&\tau^2 B_z(\psi_m,v)+\xi_0^{-1}\epsilon_p^{-1}(\xi_0)\tau^2\int_{Y\setminus R}\big[\sum^{m-1}_{l=1}(-i)^l\xi_l\nabla \psi_{m-l}\cdot \nabla \overline{ v}\nonumber \\&&+\hat{\kappa}\cdot\sum_{l=0}^{m-1}(-i)^l\xi_l(\psi_{m-1-l}\nabla \overline{v}-\nabla \psi_{m-1-l}\overline{v})-\sum_{l=0}^{m-2}(-i)^l\xi_l\psi_{m-2-l}\overline{v}\big]\nonumber \\&&-\xi_0^{-1}\epsilon_p^{-1}(\xi_0)\tau^2\epsilon_r\frac{\omega_p^2}{c^2}\int_H\big[\hat{\kappa}\cdot(\psi_{m-1}\nabla \overline{v}-\nabla \psi_{m-1}\overline{v})-\psi_{m-2}\overline{v}\big]\nonumber
\\&&-\xi_0^{-1}\epsilon_p^{-1}(\xi_0)\int_R\big[\sum^{m-2}_{l=0}(-i)^l\xi_l\nabla \psi_{m-2-l}\cdot \nabla \overline{ v}\nonumber \\&&+\hat{\kappa}\sum_{l=0}^{m-3}(-i)^l\xi_l(\psi_{m-3-l}\nabla \overline{v}-\nabla \psi_{m-3-l}\overline{v})
-\sum_{l=0}^{m-4}(-i)^l\xi_l\psi_{m-4-l}\overline{v}\big]\nonumber
\\&&+\xi_0^{-1}\epsilon_p^{-1}(\xi_0)\int_R\epsilon_r\frac{\omega_p^2}{c^2}\big[\nabla \psi_{m-2}\cdot \nabla \overline{ v}+\hat{\kappa}(\psi_{m-3}\nabla \overline{v}-\nabla \psi_{m-3}\overline{v})+\psi_{m-4}\overline{v}\big]\nonumber
\\&&-\xi_0^{-1}\epsilon_p^{-1}(\xi_0)\int_Y\big[\sum_{l=0}^{m-2}\sum_{n=0}^l \xi_{m-2-l}\xi_n \psi_{l-n} i^{l-n-m}\overline{ v}\nonumber
\\&&+\epsilon_r\frac{\omega_p^2}{c^2}\sum_{l=0}^{m-2}(-i)^l\xi_l \psi_{m-2-l}\overline{ v}\big]
=0,\quad\quad\hbox{ for all $v$ in $H^1_{per}(Y)$}.
\label{summation2}
\end{eqnarray}
Here the convention is $\psi_m=0$ for $m<0$.

The determination of  $\{\psi_m\}_{m=0}^\infty$ proceeds iteratively. We start by determining $\psi_0$ on $Y\setminus R$, this function is used as boundary data to determine $\psi_0$ in $R$ from which we determine $\psi_1$ on $Y\setminus R$ and the full sequence is determined on iterating this cycle.  The elements $\xi_m$ are recovered from solvability conditions obtained by setting $v=1$ in \eqref{summation2} and proceeding iteratively. The complete algorithm together with explicit boundary value problems necessary for the determination of the sequences $\{\psi_m\}_{m=0}^\infty$, $\{\xi_m\}_{m=0}^\infty$ is described in section \ref{higherorder} and Theorem \ref{Existencesequence}. 

The existence theory for the solution of the sequence of boundary value problems is based on the sesquilinear form $B_z(u,v)$ defined for functions $u$ and $v$ belonging to $H^1_{per}(Y\setminus R)/\mathbb{C}$. Here   $H^1_{per}(Y\setminus R)/\mathbb{C}$ is the subspace of functions $u$ belonging to  $H_{per}^1(Y\setminus R)$ with zero mean $\int_{Y\setminus R} u\,dy=0$. This space is a Hilbert space with  inner product \eqref{innerproduct}.  Although in this treatment $\epsilon_P$ is given by \eqref{singleosc} we are motivated by the general case \eqref{multleosc} and proceed in full generality allowing for the possibility that $z$ can lie anywhere on the complex plane $\mathbb{C}$ including the negative real axis. Thus for each $z$ in $\mathbb{C}$ we are required to characterize the range of the map $u \mapsto B_z(u,\cdot)$ viewed as a linear transformation $T_z$ mapping $u$ into the space of bounded skew linear functionals on $H^1_{per}(Y\setminus R)/\mathbb{C}$. This  is linked to the following eigenvalue problem  characterizing all pairs $\lambda$ in $\mathbb{C}$, $\psi$ in $H^1_{per}(Y\setminus R)/\mathbb{C}$ that solve
\begin{eqnarray}
-\frac{1}{2}\int_P\nabla\psi\cdot\nabla \overline{v}\,d\y+\frac{1}{2}\int_H\nabla\psi\cdot\nabla\overline{v}\,d\y=(\lambda\psi,v),
\label{resonance}
\end{eqnarray}
for every $v$ in $H^1_{per}(Y\setminus R)/\mathbb{C}$.
Inspection shows that for $z=(\lambda+1/2)/(\lambda-1/2)$ that $B_z(\psi,v)=0$, for every $v$ in $H^1_{per}(Y\setminus R)/\mathbb{C}$. In other words  the kernel of the operator $T_z$ is nonempty for $z=(\lambda+1/2)/(\lambda-1/2)$. 
The eigenvalues $\lambda$ will be referred to as generalized electrostatic resonances. In what follows we show that $T_z$ is a one to one and onto  map from $H_{per}^1(Y\setminus R)/\mathbb{C}$ into the dual space provided that $z\not=(\lambda+1/2)/(\lambda-1/2)$.

The generalized electrostatic resonances are characterized by introducing a suitable orthogonal  decomposition of $H^1_{per}(Y\setminus R)/\mathbb{C}$.
We introduce the subspace $H^1_0(P)$ given by the closure in the $H^1$ norm of smooth functions $v$ with compact support on $P$ and  the subspace $H^1_{0,per}(H)$ given by the closure in the $H^1$ norm of all periodic continuously differentiable functions with support outside $P$. Extending elements of $H^1_{0,per}(H)$ by zero to $Y\setminus R$ delivers  $W_1\subset H^1_{per}(Y\setminus R)$, extending elements of $H^1_0(P)$ by zero to $Y\setminus R$ delvers $W_2\subset  H^1_{per}(Y\setminus R)$. Define $W_3$ to be all functions  $w$ in $H^1_{per}(Y\setminus R)$ for which the boundary integral $\int_{\partial P}\, w\,dS$ vanishes and that belong to the orthogonal complement of $W_1\cup W_2$ with respect to the inner product \eqref{innerproduct}. Its easily verified that $W_1$, $W_2$, and $W_3$ are pairwise orthogonal with respect to the inner product \eqref{innerproduct} and
\begin{eqnarray}
H^1_{per}(Y\setminus R)/\mathbb{C}=W_1\oplus W_2\oplus W_3\oplus \mathbb{C},
\label{decomppp}
\end{eqnarray}
where the constant part of a function $u$ belonging to this space is uniquely determined by the condition $\int_{Y\setminus R} u\,dy=0$.

The following theorem 
describing all eigenvalue eigenfunction pairs is established in section \ref{genrealizedelectrostaticspectra}.
\begin{theorem}
The eigenvalues for \eqref{resonance} are real and constitute a denumerable set contained inside $[-1/2,1/2]$ with the only accumulation point being zero. The eigenspaces associated with $\lambda=1/2$ and  $\lambda=-1/2$ are $W_1$ and $W_2$ respectively. Eigenspaces associated with distinct eigenvalues in $(-1/2,1/2)$ are  pairwise orthogonal, and their union spans the subspace $W_3$.  
\label{completeeigen}
\end{theorem}

We denote the denumerable set of eigenvalues for \eqref{resonance} by the sequence  $\{\lambda_i\}_{i=1}^\infty$. Here we put $\lambda_1=0$. The orthogonal projections onto $W_1$ and $W_2$ are denoted by $\mathcal{P}_1$ and $\mathcal{P}_2$. The orthogonal projections associated with $\lambda_i$ in $(-1/2,1/2)$ are denoted by $\mathcal{P}_{\lambda_i}$. 
For $z\not=(\lambda_i+1/2)/(\lambda_i-1/2)$ the following existence theorem holds.
\begin{theorem}
Suppose $z\neq (\lambda_i+1/2)/(\lambda_i-1/2)$ for $\lambda_i \in [-\frac{1}{2},\frac{1}{2}]$ then
\begin{itemize}
\item
For any $F\in \left[H^1_{per}(Y\setminus R)/\mathbb{C}\right]^*$ such that $F(v)=0 $ for constant $v$, there exists a unique solution  $u\in H^1_{per}(Y\setminus R)/\mathbb{C} $ of the variational problem $B_z(u,v)=\overline{F(v)}$ for all $v\in H^1_{per}(Y\setminus R)/\mathbb{C}$. 

\item The transformation $T_z$ from $H^1_{per}(Y\setminus R)/\mathbb{C}$ onto itself  has the representation formula given by
\begin{eqnarray}
T_z =\mathcal{P}_1 +z \mathcal{P}_2 +\sum_{-\frac{1}{2}<\lambda_n<\frac{1}{2}} (1+(z-1)(\frac{1}{2}-\lambda_n))\mathcal{P}_{\lambda_n},
\label{Tz1}
\end{eqnarray}
with inverse 
\begin{eqnarray}
T_z^{-1}=\mathcal{P}_1 +z^{-1} \mathcal{P}_2 +\sum_{-\frac{1}{2}<\lambda_n<\frac{1}{2}} (1+(z-1)(\frac{1}{2}-\lambda_n))^{-1}\mathcal{P}_{\lambda_n},\label{Tzinus1}
\end{eqnarray}
and $B_z(u,v)=(T_z u,v)$ for all $u$, $v$ in $H^1_{per}(Y\setminus R)/\mathbb{C}$.
\end{itemize}
\label{existencetheoremforoutsideR}
\end{theorem}
\noindent This theorem is proved in section \ref{exist}.

In the following sections we present the sequence of boundary value problems for determining $\psi_m$ in $Y\setminus R$ and $R$ together with the sequence of solvability conditions characterizing $\xi_m$. In this section we use the complete orthonormal systems of eigenfunctions associated with electrostatic resonances and Dirichlet eigenvalues to explicitly solve for fields $\psi_0$ in $Y$ and $\psi_1$ in $Y\setminus R$ and provide an explicit formula for $\xi_0$.
In the following section  the boundary value problems used to determine $\psi_m$ in $Y\setminus R$, $m\geq 2$ and $\psi_m$ in $R$ for $m\geq 1$
together with solvability conditions for $\xi_m$ for $m\geq 1$ are shown to be a well posed infinite system of equations see Theorem \ref{Existencesequence}.

Applying the convention $\psi_m=0$ for $m<0$ in \eqref{summation2} shows that $\psi_0$ is the solution of
\begin{eqnarray}
B_z(\psi_0,v)=0,\hbox{ for all $v$ in $H^1(Y\setminus R)$}.
\label{zero}
\end{eqnarray}
From Theorem  \ref{completeeigen} and \eqref{decomppp} we have the  dichotomy:
\begin{enumerate}
\item   $\xi_0$ satisfies $\epsilon_P^{-1}(\xi_0)=(\lambda_i+1/2)/(\lambda_i-1/2)$ and $\psi_0$ is an eigenfunction for \eqref{resonance}.
\item    $\xi_0$ satisfies $\epsilon_P^{-1}(\xi_0)\not=(\lambda_i+1/2)/(\lambda_i-1/2)$, $i=1,2,\ldots$ and $\psi_0=constant$.
\end{enumerate}
In this article we assume the second alternative. Subsequent work will investigate the case when the first alternative is applied.
For future reference the condition $\epsilon_P^{-1}(\xi_0)\not=(\lambda_i+1/2)/(\lambda_i-1/2)$ is equivalent to
\begin{eqnarray}
\xi_0\not=\zeta_i,&\zeta_i\equiv(\lambda_i+\frac{1}{2})\frac{\epsilon_r\omega_p^2}{c^2}&,0\leq\zeta_i\leq\frac{\epsilon_r\omega_p^2}{c^2} \hbox{ $i=1,2,\ldots$.}
\label{notonresonance}
\end{eqnarray}

Restricting to test functions $v$ with support in $R$ in \eqref{summation2}   we get
\begin{eqnarray}
 \int_R\left(\nabla \psi_0\cdot\nabla \overline{v}-\xi_0\psi_0 \overline{v}\right)\,d\y=0.
\label{weakformofu0inR}
\end{eqnarray}
From continuity we have the boundary condition for $\psi_0$ on   $R$ given by $\psi_0=const.$
We denote the Dirichlet eigenvalues for $R$ by $\nu_j$, $j=1,2,\ldots$. Here we have the alternative:
\begin{enumerate}
\item If $\xi_0$ is a Dirichlet eigenvalue $\nu_i$ of $-\Delta$ in $R$  then $\psi_0(\y)=0$ for $\y$ in $Y\setminus R$.
\item  If $\xi_0\not=\nu_i$, $i=1,2,\ldots$ then $\psi_0$ is the unique solution of the Helmholtz equation \eqref{weakformofu0inR} and $\psi_0=const.$ in $Y\setminus R$.
\label{altdir}
\end{enumerate}
In this treatment we will choose the second alternative $\xi_0\not=\nu_i$. The case when the first alternative is chosen will be taken up in future investigation. Since $u_0=\underline{u}_0\psi_0$ where $\underline{u}_0$ is an arbitrary constant we can without loss of generality make the choice $\psi_0=1$ for $\y$ in $Y\setminus R$.  Since $\xi_0\not=\nu_i$, $i=1,\dots$ and $\psi_0=1$ in $Y\setminus R$ a straight forward calculation gives $\psi_0$ in $R$ in terms of the complete set of  Dirichlet eigenfunctions and eigenvalues:
\begin{eqnarray}
 &&\psi_0=\sum_{n=1}^{\infty}\frac{\mu_n<\phi_n>_R}{\mu_n-\xi_0}\phi_n , \hbox{ in $R$}.
\label{explicitformofpsi0}
\end{eqnarray}
Note here that $\mu_n$ denote the  Dirichlet eigenvalues of $-\Delta$ in $R$ whose eigenfunctions $\phi_n$ have nonzero mean, $<\phi_n>_R=\int_R\phi_n(y)dy\neq 0$. The Dirichlet eigenvalues associated with zero mean eigenfunctions are denoted by $\mu'_n$ and $\{\nu_n\}_{n=1}^\infty=\{\mu_n\}_{n=1}^\infty\cup\{\mu'_n\}_{n=1}^\infty$.
\par

To find $\psi_1$ in $Y\setminus R$, we appeal to (\ref{summation2}) with $\psi_0=1$ in $Y\setminus R$ to discover
\begin{eqnarray}
B_z(\psi_1,v)=-\int_H \hat\kappa \cdot \nabla\overline{v}-\int_P\epsilon_P^{-1}(\xi_0)\hat\kappa\cdot \nabla\overline{v} ~~~\forall ~~ v\in H^1_{per}(Y)
\end{eqnarray}
It follows from  Theorem \ref{existencetheoremforoutsideR} that the problem has a unique solution subject to the mean-zero condition: $\int_{Y\setminus R}\psi_1=0$ provided that $\xi_0\not=\zeta_i$,  $i=1,2,\ldots$. 
We apply the decomposition of $H^1_{per}(Y\setminus R)/\mathbb{C}$ given by \eqref{decomppp} and represent $\psi_1$ in terms of the complete set of orthonormal eigenfunctions $\{\psi_{\lambda_n}\}\subset W_3$ associated with $-1/2<\lambda_n <1/2$ together with 
the complete orthonormal sets of functions for $W_1$ and $W_2$, denoted by $\{\psi^1_n\}_{n=1}^\infty$ and $\{\psi^2_n\}_{n=1}^\infty$ respectively.
A straight forward calculation gives the representation for $\psi_1$ in $Y\setminus R$
{\footnotesize\begin{eqnarray}
\psi_1=-\sum_{-\frac{1}{2}<\lambda_n<\frac{1}{2}}\left(\frac{(\alpha^1_{\lambda_n}+\epsilon_P^{-1}(\xi_0)\alpha^2_{\lambda_n})}{1+(\epsilon_P^{-1}(\xi_0)-1)(1-\lambda_n)}\right)\psi_{\lambda_n}+\sum_{n=1}^\infty \alpha_{1,n}\psi^1_n,\hbox{ in $Y\setminus R$}
\label{expansionpsi1}
\end{eqnarray}}
with
{\footnotesize\begin{eqnarray}
 \alpha^1_{\lambda_n}=\hat{\kappa}\cdot\int_H\nabla\psi_{\lambda_n}\,d\y, & \alpha^2_{\lambda_n}=\hat{\kappa}\cdot\int_P\nabla\psi_{\lambda_n}\,d\y, \hbox{~and}&
\alpha_{1,n}=\hat{\kappa}\cdot\int_H\nabla\psi^1_n\,d\y.
\label{coefficients}
\end{eqnarray}}

Setting  $v=1$ and $m=2$ in \eqref{summation2} we recover the solvability condition given by 
\begin{eqnarray}
&& \tau^2\int_{H\cup P}\big[ -\hat\kappa\cdot\xi_0\nabla \psi_1+\xi_0 \big]-\tau^2\epsilon_r\frac{\omega_p^2}{c^2}\int_H(-\hat\kappa\nabla \psi_1+1)\label{solvablety}
\\&&=\int_Y(\xi_0^2\psi_0-\epsilon_r\frac{\omega_p^2}{c^2}\xi_0 \psi_0)\nonumber
\end{eqnarray}
Substitution of the spectral representations for $\psi_1$ and $\psi_0$ given by \eqref{explicitformofpsi0} and \eqref{expansionpsi1} into \eqref{solvablety} delivers the quasistatic dispersion relation
\begin{eqnarray}
 \xi_0=\tau^2 n_{eff}^{-2}(\xi_0),
\label{subwavelengthdispersion1}
\end{eqnarray}
where the effective index of diffraction $n_{eff}^2$ depends upon the direction of propagation $\hat{\kappa}$ and is written
\begin{eqnarray}
n_{eff}^2(\xi_0)=\mu_{eff}(\xi_0)/\epsilon_{eff}^{-1}(\xi_0)\hat\kappa\cdot\hat\kappa.
\label{indexofrefract}
\end{eqnarray}
The frequency dependent effective magnetic permeability $\mu_{eff}$ and effective dielectric permittivity $\epsilon_{eff}$ are given by
\begin{eqnarray}
\mu_{eff}(\xi_0)=\int_Y\psi_0=\theta_H+\theta_P+\sum_{n=1}^\infty\frac{\mu_n<\phi_n>^2_R}{\mu_n-\xi_0}\label{effperm}
\end{eqnarray}
and
{\footnotesize\begin{eqnarray}
  &&\epsilon^{-1}_{eff}(\xi_0)\hat\kappa\cdot\hat\kappa =\int_{H}(I-\mathbb{P})\hat{\kappa}\cdot\hat{\kappa}\,d\y+\frac{\xi_0}{\xi_0-\frac{\epsilon_r\omega_p^2}{c^2}}\theta_P\nonumber\\
&&-\sum_{-\frac{1}{2}<\lambda_h <\frac{1}{2}}\left( \frac{\left(\xi_0-\frac{\epsilon_r\omega_p^2}{c^2} \right)|\alpha_{\lambda_h}^{(1)}|^2+2\frac{\epsilon_r\omega_p^2}{c^2}\alpha_{\lambda _h}^{(1)}\alpha_{\lambda_h}^{(2)}+\frac{\left(\frac{\epsilon_r\omega_p^2}{c^2}\right)^2}{\xi_0-\frac{\epsilon_r\omega_p^2}{c^2}}|\alpha_{\lambda_h}^{(2)}|^2}{\xi_0-s_h} \right),\label{effdielectricconst}
\end{eqnarray}}
where $\theta_H$ and $\theta_P$ are the areas occupied by regions $H$ and $P$ respectively. The first  term on the right hand side of \eqref{effdielectricconst} is positive and frequency independent. It is written in terms of
the spectral projection $\mathbb{P}$ of square integrable vector fields over $H$ onto the subspace of gradients of potentials $\psi$ in $H_{0,per}^1(H)$.
Here $\int_{H}\mathbb{P}\hat{\kappa}\cdot\hat{\kappa}\,d\y<\theta_H$ and $\mathbb{P}\sigma=\sum_{n=1}^\infty\left(\int_H\nabla\psi_n^1\cdot\sigma\,d\y\right)\nabla\psi_n^1$ with $\sum_{n=1}^\infty|\alpha_{1,n}|^2=\int_{H}\mathbb{P}\hat{\kappa}\cdot\hat{\kappa}\,d\y$. The poles $s_h$ are the subset of values $\zeta_h$ for which at least one of the weights $\alpha_{\lambda_h}^{(1)}$ and $\alpha_{\lambda_h}^{(2)}$ are non zero. The values of $\zeta_h$ for which both weights vanish are denoted by
$s'_h$ and $\{\zeta_h\}_{h=1}^\infty=\{s_h\}_{h=1}^\infty\cup\{s'_h\}_{h=1}^\infty$. The graphs 
of $\mu_{eff}$ and $\epsilon^{-1}_{eff}\hat\kappa\cdot\hat\kappa$ as functions of $\xi_0$ are displayed in Fig. \ref{ximueffrelation} and Fig. \ref{xiepsiloneffrelation}. Here the intervals $a'<\xi_0 <b'$ and $a''<\xi_0 <b''$ are the same in all graphs.

\par
\begin{figure}[h!]
\begin{center}
\begin{psfrags} 
\psfrag{mu}{$\mu_{eff}$} 
\psfrag{wc}{$\xi_0$}
\psfrag{k1}{$\mu_1$}
\psfrag{k2}{$\mu_2$}
\psfrag{k3}{$\mu_3$}
\psfrag{c}{$\cdots$}
\psfrag{a1}{$a'$} 
\psfrag{b1}{$b'$}
\psfrag{a2}{$a''$}
\psfrag{b2}{$b''$}
\includegraphics[width=0.4\textwidth]{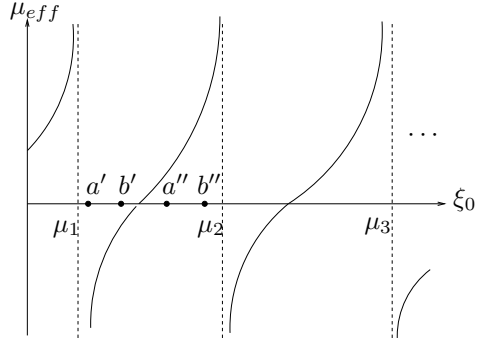}
\end{psfrags} 
\caption{The relation between $\mu_{eff}$ and $\xi_0$.}
\label{ximueffrelation}
\end{center}
\end{figure}
\par

\begin{figure}[h!]
\begin{center}
\begin{psfrags} 
\psfrag{e}{$\epsilon^{-1}_{eff}\hat\kappa\cdot\hat\kappa$} 
\psfrag{wc}{$\xi_0$}
\psfrag{k3}{$\epsilon_r\frac{\omega_p^2}{c^2}$}
\psfrag{k2}{$s_1$}
\psfrag{k1}{$s_2$}
\psfrag{c}{$\cdots$}
\psfrag{a1}{$a'$} 
\psfrag{b1}{$b'$}
\psfrag{a2}{$a''$}
\psfrag{b2}{$b''$}
\includegraphics[width=0.45\textwidth]{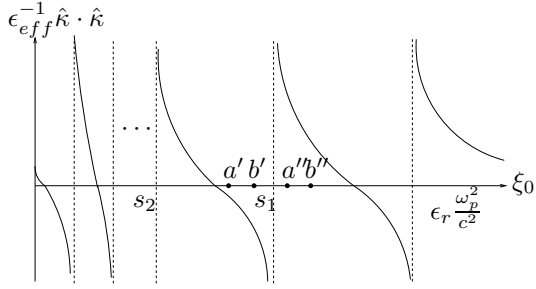}
\end{psfrags} 
\caption{The relation between $\epsilon^{-1}_{eff}\hat\kappa\cdot\hat\kappa$ and $\xi_0$.}
\label{xiepsiloneffrelation}
\end{center}
\end{figure}
\par

For future reference is is convenient to write the dispersion relation explicitly in terms of $\mu_{eff}$ and $\epsilon_{eff}$ and 
\begin{eqnarray}
 \mu_{eff}(\xi_0)\xi_0=\tau^2\epsilon_{eff}^{-1}(\xi_0)\hat\kappa\cdot\hat\kappa.
\label{subwavelengthdispersion1alternate}
\end{eqnarray}
The branches of the quasistatic dispersion relation $(\tau,\hat{\kappa})\mapsto\xi_0$  are controlled by the poles and zeros of $\mu_{eff}$ and $\epsilon_{eff}^{-1}$ explicitly
determined through the Dirichlet spectra and generalized electrostatic resonances. 

\section{Solution of higher order problems}
\label{higherorder}
Having determined $\psi_0$ in $Y$, $\psi_1$ in $Y\setminus R$, and $\xi_0$ we now
provide the algorithm for determining the rest of the elements in the sequences $\{\psi_m\}_{m=0}^\infty$, $\{\xi_m\}_{m=0}^\infty$.
The algorithm is summarized in Theorem \ref{Existencesequence}.

\noindent Step I. {\em Solution of $\psi_m$ in $R$ for $m\geq 1$.}

We restrict the trial space to test functions $v$ with support in $R$
and put $m\mapsto m+2$ in \eqref{summation2}. We decompose $\psi_m$ according to
\begin{eqnarray}
\psi_m=\tilde{\psi}_m+(-i)^{m}\xi_{m}\psi_*
\label{decomp}
\end{eqnarray}
and substitute into \eqref{summation2}.
This decomposition is chosen such that $\tilde{\psi}_m$ depends on $\xi_n$ and $\psi_n$ inside $R$ for  $n\leq m-1$. The function $\psi_*$ depends only upon $\xi_0$ and $\psi_0$ in $R$. 
The function $\tilde{\psi}_m$ solves the Dirichlet boundary value problem with $\tilde{\psi}_m|_{\partial R-}=\psi_m|_{\partial R+}$ and 
\begin{eqnarray}
\int_R\left(\nabla\tilde{\psi}_m\cdot\nabla\overline{v}-\xi_0\tilde{\psi}_m\overline{v}\right)\,d\y=\int_R F\cdot\nabla\overline{v}\,d\y+\int_R G\overline{v}\,d\y,
\label{Rproblem}
\end{eqnarray}
where
\begin{eqnarray}
F&=&-\xi_0^{-1}\epsilon_p^{-1}(\xi_0)[\sum^{m-1}_{l=1}(-i)^l\xi_l\nabla \psi_{m-l}+\hat{\kappa}(\sum_{l=0}^{m-1}(-i)^l\xi_l\psi_{m-1-l})]
\nonumber
\\&&+\xi_0^{-1}\epsilon_p^{-1}(\xi_0)\epsilon_r\frac{\omega_p^2}{c^2}\hat{\kappa}\psi_{m-1},
\label{F}
\end{eqnarray}
and
{\small
\begin{eqnarray}
G&=&\xi_0^{-1}\epsilon_p^{-1}(\xi_0)(\hat{\kappa}\cdot \sum_{l=0}^{m-1}(-i)^l\xi_l\nabla \psi_{m-1-l}+\sum_{l=0}^{m-2}(-i)^l\xi_l\psi_{m-2-l})\nonumber
\\&&-\xi_0^{-1}\epsilon_p^{-1}(\xi_0)(\hat{\kappa}\cdot\nabla \psi_{m-1}+\psi_{m-2})\label{G}
\\&&+\xi_0^{-1}\epsilon_p^{-1}(\xi_0)[\sum_{l=1}^{m-1}\sum_{n=0}^l \xi_{m-l}\xi_n\psi_{l-n}i^{l-n-m}-\epsilon_r\frac{\omega_p^2}{c^2}\sum_{l=1}^{m-1}(-i)^l\xi_l\psi_{m-l}\nonumber
\\&&+\xi_0\sum_{n=1}^{m-1}\xi_n\psi_{m-n}i^{-n}].
\nonumber
\end{eqnarray}}
This boundary value problem has a unique solution for $\xi_0\not=\nu_i$, $i=1,\ldots$.
The function $\psi_*$ is the solution of 

\begin{equation}
\begin{cases}
\int_R(-\nabla\psi_*\cdot \nabla \overline{v}+ \xi_0\psi_*\overline{v})+\int_R\psi_0\overline{v}=0\\
\psi_*|_{\partial R-}=0\label{inproblem}
\end{cases}
\end{equation}
The explicit representation for  $\psi_*$ is obtained using the Dirichlet eigenfunctions on $R$ and is given by, 
\begin{equation}
 \psi_*= \sum_{n=1}^{\infty}\frac{\mu_n<\phi_n>_R}{(\mu_n-\xi_0)^2}\phi_n  ~~\text{ in } R. 
\label{psi*}
\end{equation}

\noindent Step II. {\em Solution of $\psi_m$ in $Y\setminus R$ for $m\geq 2$.}

We decompose $\psi_m$
\begin{eqnarray}
\psi_m=\psi_m'+(-i)^{m-1}\xi_{m-1}\hat{\psi}
\label{decompp}
\end{eqnarray}
and substitute into \eqref{summation2}.
This decomposition is chosen such that $\psi'_m$ depends on $\xi_n$ for $0\leq n \leq m-2$, the functions $\psi_n$, $n\leq m-1$ in $Y\setminus R$ and the functions $\psi_n$, $n\leq m-2$ in $R$. The function $\hat{\psi}$ depends only on $\xi_0$ and $\psi_1$ in $Y\setminus R$.

For $m\geq 2$, $\psi_m'$ in $Y\setminus R$ subject to the mean zero condition $\int_{Y\setminus R} \psi_m'\,dy=0$ is the solution of
{\footnotesize
\begin{eqnarray}
\tau^2B_z(\psi_m',v)&=&\int_{Y\setminus R} \left(F_1\cdot\nabla \overline{v}+G_1\overline{v}\right)\,d\y\nonumber
\\&&+\int_R \left(F_2\cdot\nabla\overline{v}+G_2\overline{v}\right)\,d\y, \quad \hbox{ for all $v$ in $H^1_{per}(Y)$},
\label{varformpsimprime}
\end{eqnarray}}
where 
\begin{eqnarray}
F_1&=&\xi_0^{-1}\epsilon_p^{-1}(\xi_0)\tau^2\big[\sum^{m-2}_{l=1}(-i)^l\xi_l\nabla \psi_{m-l} +(\sum_{l=0}^{m-2}(-i)^l\xi_l\psi_{m-1-l})\hat{\kappa}\big]\nonumber\\
&&-\xi_0^{-1}\epsilon_p^{-1}(\xi_0)\tau^2\epsilon_r\frac{\omega_p^2}{c^2}\chi_H\hat{\kappa}\psi_{m-1},
\label{F1}
\end{eqnarray}
{\footnotesize
\begin{eqnarray}
G_1&=&\xi_0^{-1}\epsilon_p^{-1}(\xi_0)\tau^2\big[ \hat{\kappa}\cdot\sum_{l=0}^{m-2}(-i)^l\xi_l(-\nabla \psi_{m-1-l})-\sum_{l=0}^{m-2}(-i)^l\xi_l\psi_{m-2-l}\big]\nonumber\\ &&-\xi_0^{-1}\epsilon_p^{-1}(\xi_0)\tau^2\epsilon_r\frac{\omega_p^2}{c^2}\chi_H\big[\hat{\kappa}\cdot(-\nabla \psi_{m-1})-\psi_{m-2}\big]\nonumber
\\&&+\xi_0^{-1}\epsilon_p^{-1}(\xi_0)\big[\sum_{l=0}^{m-2}\sum_{n=0}^l \xi_{m-2-l}\xi_n \psi_{l-n} i^{l-n-m}\label{G1}
\\&&-\epsilon_r\frac{\omega_p^2}{c^2}\sum_{l=0}^{m-2}(-i)^l\xi_l \psi_{m-2-l}\big],\nonumber
\end{eqnarray}}
\begin{eqnarray}
F_2&=&-\xi_0^{-1}\epsilon_p^{-1}(\xi_0)\big[\sum^{m-2}_{l=0}(-i)^l\xi_l\nabla \psi_{m-2-l} +\hat{\kappa}(\sum_{l=0}^{m-3}(-i)^l\xi_l\psi_{m-3-l})
\big] \nonumber\\
&&+\xi_0^{-1}\epsilon_p^{-1}(\xi_0)\epsilon_r\frac{\omega_p^2}{c^2}\big[\nabla \psi_{m-2}+\hat{\kappa}(\psi_{m-3})\big],
\label{F2}
\end{eqnarray}
and
{\footnotesize
\begin{eqnarray}
G_2&=&-\xi_0^{-1}\epsilon_p^{-1}(\xi_0)\big[ +\hat{\kappa}(\sum_{l=0}^{m-3}(-i)^l\xi_l(-\nabla \psi_{m-3-l}))\nonumber
\\&&-\sum_{l=0}^{m-4}(-i)^l\xi_l\psi_{m-4-l}\big] +\xi_0^{-1}\epsilon_p^{-1}(\xi_0)\epsilon_r\frac{\omega_p^2}{c^2}\big[\hat{\kappa}\cdot(-\nabla \psi_{m-3})+\psi_{m-4}\big]\nonumber
\\&&+\xi_0^{-1}\epsilon_p^{-1}(\xi_0)\big[\sum_{l=0}^{m-2}\sum_{n=0}^l \xi_{m-2-l}\xi_n \psi_{l-n} i^{l-n-m}\overline{v}\label{G2}
\\&&-\epsilon_r\frac{\omega_p^2}{c^2}\sum_{l=0}^{m-2}(-i)^l\xi_l \psi_{m-2-l}\overline{ v}\big],\nonumber 
\end{eqnarray}}
where $\chi_H$ is the indicator function of the host domain $H$ taking $1$ inside and zero outside. Here we have used the identity $\psi_0(\y)=1$, for $\y\in Y\setminus R$ in determining the formulas for the right hand side of \eqref{varformpsimprime}.

Setting $v=1$ in \eqref{varformpsimprime} delivers the solvability condition
\begin{eqnarray}
\int_{Y\setminus R}\,G_1\,d\y+\int_{R}\,G_2\,d\y=0.
\label{solvabilitycondt}
\end{eqnarray}

The condition 
\begin{eqnarray}
-\nabla\cdot F_2+G_2=0\hbox{  for $\y$ in $R$}\label{divzero}
\end{eqnarray}
follows from the solution of Step I see, \eqref{decomp}, \eqref{Rproblem} and \eqref{inproblem}.
Integration by parts on the right hand side of \eqref{varformpsimprime} together with \eqref{divzero} transforms \eqref{varformpsimprime}
 into the equivalent Neumann boundary value problem for $\psi'_m$, $m\geq 2$, given by
{\footnotesize
\begin{eqnarray}
\tau^2B_z(\psi_m',v)&=&\int_{Y\setminus R} \left(F_1\cdot\nabla \overline{v}+G_1\overline{v}\right)\,d\y\label{varformpsimprimeYminusR}
\\&&+\int_{\partial R} F_2\cdot n\overline{v}\,ds, \quad\hbox{ for all $v$ in $H^1_{per}(Y\setminus R)$},\nonumber
\end{eqnarray}}
where $n$ is the outward directed unit normal to $\partial R$. Finally we observe that the solvability condition for the Neumann problem \eqref{varformpsimprimeYminusR} given by
\begin{eqnarray}
\int_{Y\setminus R}\,G_1\,d\y+\int_{\partial R}\,F_2\cdot n\,ds=0,
\label{solvabilitycondtneuman}
\end{eqnarray}
follows immediately from \eqref{solvabilitycondt} and \eqref{divzero}. This Neumann problem satisfies the hypotheses of Theorem \ref{existencetheoremforoutsideR} and we assert the existence of a  solution $\psi'_m$ for \eqref{varformpsimprimeYminusR}, uniquely determined by the condition $\int_{Y\setminus R} \psi_m'\,dy=0$,  provided that $\xi_0\not=\zeta_i$, $0\leq \zeta_i\leq\frac{\epsilon_r\omega_P^2}{c^2}$, $i=1,2,\ldots$.

The field $\hat{\psi}$ 
solves 
\begin{equation}
\begin{cases}
B_z(\hat{\psi},v)=-\xi_0^{-1}\epsilon_p^{-1}(\xi_0)\int_{Y\setminus R}(\nabla \psi_1+\hat\kappa)\cdot\nabla\overline{v}, \\
\\
\int_{Y\setminus R}\hat\psi=0 ,
\end{cases}
\label{hatpsisolution}
\end{equation}
for all trials $v\in H^1_{per}(Y\setminus R)$ .
The solution $\hat{\psi}$ is represented explicitly in terms of the eigenvectors associated with the generalized electrostatic resonances.
A straight forward calculation gives
{\footnotesize
\begin{eqnarray}
\hat{\psi}&=&\sum_{-\frac{1}{2}<\lambda_n<\frac{1}{2}}\frac{\psi_{\lambda_n}}{\xi_0-(\frac{1}{2}+\lambda_n)\frac{\epsilon_r\omega_p^2}{c^2}}\nonumber
\\&&\left(-(\alpha_{\lambda_n}^1+\alpha_{\lambda_n}^2)+\frac{(\xi_0-\frac{\epsilon_r\omega_p^2}{c^2})\left(\alpha_{\lambda_n}^1+\alpha_{\lambda_n}^2\left(\frac{\xi_0}{\xi_0-\frac{\epsilon_r\omega_p^2}{c^2}}\right)\right)}{\xi_0-s_n}\right).
\label{psihatformula}
\end{eqnarray}}

\noindent Step III. {\em Solution of $\xi_m$  for $m\geq 1$.} 

We apply the decomposition \eqref{decomp} to $\psi_{m-2}$ in $R$ and \eqref{decompp} to $\psi_{m-1}$ in $Y\setminus R$.   Substitution of the decompositions into the solvability condition \eqref{solvabilitycondt} and applying the explicit spectral representations \eqref{psi*}, \eqref{psihatformula} for $\psi_*$  and $\hat{\psi}$ we obtain the recursion relation for determining $\xi_1,\xi_2,\ldots$ given by
{\footnotesize
\begin{eqnarray}
\xi_{m-2}\mathcal{G}(\xi_0) &=&\tau^2\int_{Y\setminus R}\big[\hat{\kappa}\cdot (\sum_{l=1}^{m-3}(-i)^l\xi_l\nabla\psi_{m-1-l}+\xi_0\nabla \psi_{m-1}')+\sum_{l=0}^{m-3}(-i)^l\xi_l\psi_{m-2-l}\big]\nonumber
\\&&-\tau^2\int_H\epsilon_r\frac{\omega_p^2}{c^2}(\hat\kappa\cdot\nabla\psi_{m-1}'+\psi_{m-2})\nonumber
\\&&-\int_R\big[\hat{\kappa}\cdot \sum_{l=0}^{m-3}(-i)^l\xi_l\nabla\psi_{m-3-l}+\sum_{l=0}^{m-4}(-i)^l\xi_l\psi_{m-4-l}\big]\label{iterationforxim1}
\\&&-\int_Y
\big[\sum_{l=1}^{m-3}\sum_{n=0}^l i^{l-n-m}\xi_{m-2-l}\xi_n\psi_{l-n}-\sum_{l=1}^{m-3}(-i)^l\xi_0\xi_l\psi_{m-2-l}\big]\nonumber
\\&&-\int_R\xi_0^2\tilde\psi_{m-2}-\int_{Y\setminus R}\xi_0^2\psi_{m-2} +\int_Y \epsilon_r\frac{\omega_p^2}{c^2}\sum_{l=1}^{m-3}(-i)^l\xi_l\psi_{m-2-l}\nonumber
\\&&+\int_R \epsilon_r\frac{\omega_p^2}{c^2}\xi_0\tilde\psi_{m-2}+\int_{Y\setminus R} \epsilon_r\frac{\omega_p^2}{c^2}\xi_0\psi_{m-2}.\nonumber
\end{eqnarray} }
Noting that \eqref{decomp} and \eqref{decompp} can also be applied to the lower order terms on the right hand side of \eqref{iterationforxim1} we see that \eqref{iterationforxim1} determines $\xi_{m-2}$ from $\xi_n$,  $\psi_n$ in $Y\setminus R$,  and $\psi_n$ in $R$ with $0\leq n\leq m-3$. Here the ``solvability matrix,'' $\mathcal{G}$ depends explicitly on $\xi_0$ and is given by
\begin{eqnarray}
\mathcal{G}(\xi_0) &=&\theta_H+\theta_P\nonumber
\\&&+\tau^2\big[\int_H\mathbb{P}\hat{\kappa}\cdot\hat{\kappa}\,d\y-\sum_{n=1}^\infty (g_n^2(\xi_0)|\alpha_{\lambda_n}^1+\epsilon^{-1}_P(\xi_0)\alpha^2_{\lambda_n}|^2)\nonumber
\\&&+2g_n(\xi_0)(\alpha_{\lambda_n}^1+\alpha^2_{\lambda_n})(\alpha_{\lambda_n}^1+\epsilon^{-1}_P(\xi_0)\alpha^2_{\lambda_n})\big]\label{B}
\\&& -2(\xi_0+\frac{\epsilon_r\omega_p^2}{c^2})\mu_{eff}(\xi_0)+\xi_0(\xi_0-\frac{\epsilon_r\omega_p^2}{c^2})\sum_{n=1}^\infty\frac{\mu_n\langle\phi_n\rangle_R^2}{(\mu_n-\xi_0)^2},\nonumber
\end{eqnarray} 
where $g_n(\xi_0)=(\xi_0-\frac{\epsilon_r\omega_p^2}{c^2})/(\xi_0-s_n)$ and $\epsilon_P^{-1}(\xi_0)=\xi_0/(\xi_0-\frac{\epsilon_r\omega_p^2}{c^2})$.
The set of poles and zeros for $\mathcal{G}(\xi_0)$ are explicitly controlled by the generalized electrostatic resonances and Dirichlet spectra. The zeros of $\mathcal{G}$  are denoted by $\gamma_n$, $n=1,\ldots$. It is now evident that the recursion relation \eqref{iterationforxim1}  holds provided that $\xi_0\neq \nu_n$,  $\xi_0\neq \zeta_n$, and $\xi_0\not=\gamma_n$.

Denote the union of the Dirichlet spectra, the electrostatic resonances, and zeros of $\mathcal{G}$ by
\begin{eqnarray}
U\equiv\{\nu_j\}_{j=1}^\infty\cup\{\zeta_j\}_{j=1}^\infty\cup\{\gamma_j\}_{j=1}^\infty
\label{unionofspectra}
\end{eqnarray}
where $\{\nu_j\}_{j=1}^\infty\subset \mathbb{R}^+$, $\{\gamma_j\}_{j=1}^\infty\subset \mathbb{R}^+$, and $\{\zeta_j\}_{j=1}^\infty\subset [0,\frac{\epsilon_r\omega_p^2}{c^2}]$.
We collect results and state the following theorem.

\begin{theorem}
 If $\xi_0$ belongs to $\mathbb{R}^+\setminus U$ then the sequences  $\{\psi_m\}_{m=0}^\infty\subset  H^1_{per}(Y)$, $\{\xi_m\}_{m=0}^\infty\subset \mathbb{C}$ exist and are uniquely determined.
 \label{Existencesequence}
\end{theorem}

\begin{proof}
The functions $\psi_0$ in $H^1(Y)$ and $\psi_1$ in $H^1(Y\setminus R)$ with $\int_{Y\setminus R}\psi_1\,dy=0$ exist as does $\xi_0$. Application of Step I uniquely determines $\tilde{\psi_1}$ in $R$. Application of Step II uniquely determines $\psi'_2$ in $Y\setminus R$. Step III uniquely determines $\xi_1$ and we apply \eqref{decomp} and \eqref{decompp} to recover $\psi_1$ in $H^1(Y)$ and $\psi_2$ in  $H^1(Y\setminus R)$ with $\int_{Y\setminus R}\psi_2\,dy=0$.
We now adopt the induction hypotheses $H_n$, $n\geq 2$:\\
There exist\\
1) $\psi_n\in H^1_{per}(Y\setminus R)$ with $\int_{Y\setminus R}\psi_n=0$ ;\\
2) $\psi_{j}\in H^1_{per}(Y) $ for $0\leq j\leq n-1$ with $\int_{Y\setminus R}\psi_{j}=0$, $1\leq j\leq n-1$;\\
3) $\xi_{n-1} \in \mathbb{C}$.\\
Application of Step I uniquely determines $\tilde{\psi_n}$ in $R$. Application of Step II uniquely determines $\psi'_{n+1}$ in $Y\setminus R$. Step III uniquely determines $\xi_n$ and we apply \eqref{decomp} and \eqref{decompp} to recover $\psi_n$ in $H^1(Y)$ and $\psi_{n+1}$ in  $H^1(Y\setminus R)$ with $\int_{Y\setminus R}\psi_{n+1}\,dy=0$.
\end{proof}


\section{Band structure for the metamaterial crystal}
\label{bandstructure}
In this section we present 
explicit formulas describing the band structure for the metamaterial crystal in the dynamic, sub-wavelength regime $1>\eta>0$.
The  formulas show that the  frequency intervals associated with pass bands and stop bands are governed by the poles and zeros of the effective magnetic permittivity and dielectric permittivity tensors. The poles and zeros are explicitly determined  by the Dirichlet spectra of $R$ and the generalized electrostatic resonances of $Y\setminus R$, see \eqref{effperm} and \eqref{effdielectricconst}. For the general class of  parallel rod configurations  treated here the pass bands and stop bands can exhibit anisotropy, i.e.,  dependence on the direction of propagation described by $\hat{\kappa}$. The anisotropy is governed
by the projection of $\hat{\kappa}$ onto the eigenfunctions associated with the generalized electrostatic resonances given by \eqref{coefficients}.
In what follows pass bands are explicitly linked to frequency intervals and propagation directions for which both the effective magnetic permittivity and dielectric permeability have the same sign.  This includes frequency intervals where effective tensors are either simultaneously positive or negative.

\par
\begin{figure}[h!]
\begin{center}
\begin{psfrags} 
\psfrag{a1}{$a'$} 
\psfrag{b1}{$b'$}
\psfrag{a2}{$a''$}
\psfrag{b2}{$b''$}
\psfrag{tao}{$\tau^2$}
\psfrag{xi}{$\xi_0$}
\psfrag{I1}{$I_{n'}$}
\psfrag{I2}{$I_{n''}$}
\includegraphics[width=0.4\textwidth]{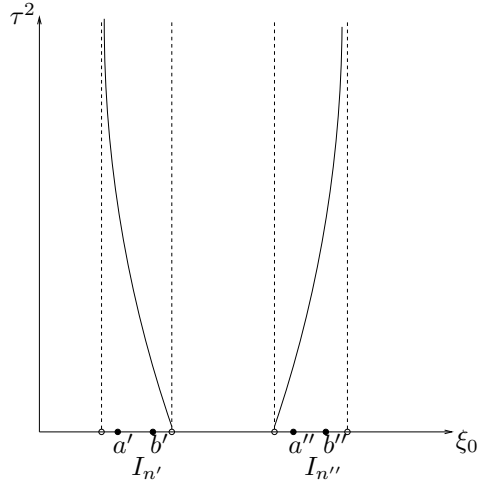}
\end{psfrags} 
\caption{The leading order dispersion relation over two selected intervals $I_{n'}$ and $I_{n''}$.}
\label{tauxirelation}
\end{center}
\end{figure}
\par

We begin by identifying the locations of the branches for the dispersion relation associated with the metamaterial crystal. We introduce the union of open intervals  $\bigcup_n O_n$ on the positive real axis $\mathbb{R}^+$ obtained by removing the points $\{\zeta_j\}_{j=1}^\infty$, $\{\nu_j\}_{j=1}^\infty$, $\{\gamma_j\}_{j=1}^\infty$, $\{s_j^*\}_{j=1}^\infty$ and  $\{\mu_j^*\}_{j=1}^\infty$.   Here $s_j^*$ is the zero of $\epsilon_{eff}^{-1}(\xi_0)\hat\kappa\cdot\hat\kappa$ between $s_j$ and $s_{j+1}$ and  $\mu_j^*$ is the zero of $\mu_{eff}(\xi_0)$ between $\mu_j$ and $\mu_{j+1}$.  The leading order dispersion relation \eqref{subwavelengthdispersion1alternate}
implicitly defines  the map $(\tau,\hat\kappa)\mapsto\xi_0$.  We denote the branch associated with $O_n$ by $\xi_0^n=\xi_0^{(n)}(\tau,\hat\kappa)$ . Let $I_n$ be an open interval strictly contained inside $O_n$.  For this choice $I_n\subset O_n$ does not intersect the union of the Dirichlet spectra, generalized electrostatic spectra and the zeros of the solvability matrix $\mathcal{G}$, see  \eqref{unionofspectra}.  

The band structure for the metamaterial in the dynamic sub-wavelength  regime $1>\eta>0$ is given by the following theorem.
\par
\begin{theorem}
The dispersion relation for the metamaterial crystal contains an infinite sequence of branches with leading order behavior given by the 
functions $\xi_0^n=\xi_0^n(\tau,\hat{\kappa})$. For $\tau$ such that $\xi_0^n$ belongs to $I_n$  there exists a constant $R$ depending only on  $I_n$ such that for $0<\rho< R$, the branch of the dispersion relation for the metamaterial crystal is given by
\begin{eqnarray}
 \xi=\xi_0^n(\tau,\hat\kappa)+\sum_{l=1}^{\infty}(\rho\tau)^{l}\xi_{l}^n,
 \label{xipowerseries}
\end{eqnarray}
for
\begin{eqnarray}
\{-2\pi\leq \tau\rho\hat{\kappa}_1\leq 2\pi , -2\pi\leq \tau\rho\hat{\kappa}_2\leq 2\pi\}.
\label{brullionfortau}
\end{eqnarray}
Here the higher order terms $\xi_{l}^n$ are real and are uniquely determined by $\xi_0^n$ according to Theorem \ref{Existencesequence}.
\label{convergenceofxi}
\end{theorem}

The power series representation for the transverse magnetic Bloch wave  \eqref{em3}  is given by the following theorem.
\begin{theorem}
For $I_n$ and $R$ as in Theorem \ref{convergenceofxi} and for $\tau$ such that $\xi_0^n$ belongs to $I_n$  there exist transverse magnetic  Bloch waves given by the expansion
\begin{eqnarray}
H_3=\underline u_0\left(\psi_0(\x/d)+\sum_{l=1}^{\infty}(\rho\tau)^li^l\psi_l(\x/d)\right)exp\left\{i\left(k\hat\kappa\cdot \x-t\omega\right)\right\},
\end{eqnarray}
where  the series
\begin{eqnarray}
\psi_0(\y)+\sum_{l=1}^{\infty}(\rho\tau)^li^l\psi_l(\y)
\label{convgncee}
\end{eqnarray}
is summable in  $H^1(Y)$  for $0\leq\rho<R$ and $\frac{\omega}{c}=\sqrt{\frac{\xi}{\epsilon_r}}$.
\label{summable}
\end{theorem}

Theorem \ref{convergenceofxi} explicitly shows that the leading order behavior  determines the existence of pass bands or stop bands when $\rho=d/\sqrt{\epsilon_r}$ is sufficiently small. With this in mind we can appeal to \eqref{subwavelengthdispersion1alternate} and state the following theorem. 
\begin{theorem}
\label{bands}
For $\rho>0$ sufficiently small: 
\begin{itemize}
\item There exist propagating Bloch wave solutions along directions $\hat{\kappa}$ over intervals $I_n$  for which $\mu_{eff}(\xi_0^n)$ and $\epsilon_{eff}^{-1}(\xi_0^n)\hat{\kappa}\cdot\hat{\kappa}$ have the same sign. 
\item 
Intervals $I_n$ for which $\mu_{eff}(\xi_0^n)$ and $\epsilon_{eff}^{-1}(\xi_0^n)\hat{\kappa}\cdot\hat{\kappa}$ have the opposite sign lie within stop bands.
\end{itemize}
\end{theorem}

Leading order behavior for pass bands associated with double negative and double positive behavior are illustrated in Fig. \ref{tauxirelation} for two intervals $I_{n'}$ and $I_{n''}$.  Here the effective properties are both negative over the interval $I_{n'}$ while they are both positive over $I_{n''}$.

\section {Generalized electrostatic spectra}
\label{genrealizedelectrostaticspectra}

In this section we establish Theorem \ref{completeeigen} and characterize all pairs $(\lambda,\psi)$ in $\mathbb{C}\times H_{per}^1(Y\setminus R)$ satisfying  \eqref{resonance}. We start by recalling the bilinear form in \eqref{resonance} and forming the quotient
\begin{eqnarray}
Q(u)=\frac{-\frac{1}{2}\int_P|\nabla u|^2\,dx+\frac{1}{2}\int_H|\nabla u|^2\,dx}{(u,u)}.
\label{quotient}
\end{eqnarray}
From \eqref{quotient} it is evident that $\lambda$ lies inside $[-1/2,1/2]$.
It is easily seen that the solutions $(\lambda,\psi)$ of \eqref{resonance} for the choices $\lambda=\frac{1}{2}$ and $\lambda=-\frac{1}{2}$ correspond to $\psi$ in $W_1$ and $W_2$ respectively. Only the $0$ element of $W_3$ satisfies \eqref{resonance} for the choices $\lambda=\frac{1}{2}$ and $\lambda=-\frac{1}{2}$.

We now investigate solutions $(\lambda,\psi)$ for $\psi$ belonging to $W_3$.
The bilinear form \eqref{resonance} defines a map $T$ from $W_3$ into the space of skew linear functionals on $W_3$. Here $W_3$ is a Hilbert space and 
\begin{eqnarray}
(Tu,v)=-\frac{1}{2}\int_P\nabla u\cdot\nabla \overline{v}\,d\y+\frac{1}{2}\int_H\nabla u\cdot\nabla\overline{v}\,d\y,
\label{resonanceT}
\end{eqnarray}
for $u$ and $v$ in $W_3$.
In what follows we show that $T$ is a compact map from $W_3$ onto $W_3$ with eigenvalues contained inside the open interval $(-1/2,1/2)$.
We do this by providing an explicit representation for $T$ given in terms of a composition of single and double layer potentials.

Set $\mathbb{D}=\cup_{n\in\mathbb{Z}^2}((Y\setminus R)+n)$ and introduce the Green's function $G(\x,\y)$  on $\mathbb{D}$ such that:
$G$ is separately periodic in $\x$ and $\y$  with period $Y$, $G$ is $C^2$ in each of the variables $\x$ and $\y$ for $\x\not=\y$, and
{\footnotesize
\begin{eqnarray}
\Delta_\x G(\x,\y)&=&\sum_{n\in \mathbb{Z}^2}\delta(\x-\y+n) -1\quad\text{ in } \mathbb{D},\nonumber
\\\partial_{n_\x}G(\x,\y)&=&\frac{|R|}{|\partial R|},\quad \text{ for } \x\in \cup_{n\in\mathbb{Z}^2} (\partial R+n),
\label{greensfunction}
\end{eqnarray}}
where $\delta_\x$ is the Dirac delta function located at $\x$ and $|R|$ and $|\partial R|$ are the area and arc-length of $R$ and $\partial R$ respectively. 
We can find $G(\x,\y)$  as a sum of the periodic free space Green's function $F$ for the Laplacian and a corrector $\phi^*$: $G(\x,\y)=F(\x,\y)+\phi^*(\x,\y)$ , where 
\begin{eqnarray}
 F(x,y)=-\sum_{n\in \mathbb Z^2\setminus\{0\}}\frac{e^{i2\pi n\cdot(x-y)}}{4\pi^2|n|^2},
 \label{freespacegreens}
\end{eqnarray}
$\phi^*(\x,\y)$ is periodic in $\x$ and $\y$ with period $Y$, $C^2$ in $\x$, $\y$,  and solves
\begin{eqnarray}
&&
 \Delta_\x \phi^*(\x,\y)=0 ~\hbox{ for $\x$, $\y$ in } \mathbb{D} ,\nonumber\\
 && ~~\partial_{n_\x}\phi^*(\x,\y)_{{|_{\partial R}^+}}=-\partial_{n_\x}F(\x,\y)_{{|_{\partial R}^+}} +\frac{|R|}{|\partial R|}, \hbox{for $\x$ on $\partial R$} ,\label{pdeforphi}\nonumber\\
&&\hbox{and }\int_{\partial R}\phi^*(\x,\y)dS_\x=0.
 \label{correctorproblem}
\end{eqnarray}
Here the subscript ${\partial R}^+$, indicates  traces of functions on $\partial R$ taken from the outside of $R$.

The space of mean zero square integrable functions defined on $\partial P$ is denoted by $L^2_0(\partial P)$
and for $\phi$ in $L^2_0(\partial P)$ we introduce the single layer potential $\tilde{S}_P$ given by
\begin{eqnarray}
\tilde{S}_P(\phi)=\int_{\partial P} G(\x,\y)\phi(\y)dS_\y,
\label{singlelayer}
\end{eqnarray}
and the modified single layer potential $S_P$ mapping $L^2_0(\partial P)$ into $W_3$ 
given by
\begin{eqnarray}
{S}_P(\phi)=\tilde{S}_P(\phi)-|\partial P|^{-1}\int_{\partial P} \tilde{S}_P(\phi)(\y)dS_\y.
\label{singlelayerzeroavg}
\end{eqnarray}

On restricting $S_P(\phi)(\x)$ to $\x$ on $\partial P$, one  readily verifies that $S_P$ is self-adjoint with respect to the $L^2(\partial P)$ inner product and is a bounded one to one linear map from $L_0^2(\partial P)$ onto
 $H^{1/2}(\partial P)\cap L^2_0(\partial P)$. The inverse denoted by $S_{\partial P}^{-1}:H^{1/2}(\partial P)\cap L_0^2(\partial P)\rightarrow L_0^2(\partial P)$  is linear and continuous. We introduce the trace operator $\gamma$ mapping $W_3$ into $H^{1/2}(\partial P)\cap L_0^2(\partial P)$, this map is bounded and onto \cite{Adams}.
Collecting results we have the following
\begin{lemma}
 $S_p$ is a one to one, bounded linear transformation from $L^2_0(\partial P)$ onto $W_3$ with $S^{-1}_P: W_3\rightarrow L_0^2(\partial P)$  linear and continuous given by $S_P^{-1}=S_{\partial P}^{-1}\gamma$.
 \end{lemma}
 \begin{proof}
Given $u$ in $W_3$ consider its trace on $\partial P$ given  by $\gamma u=u_{{|_{\partial P}}}$ and  $u_{{|_{\partial P}}}$ belongs to $H^{1/2}(\partial P)\cap L^2_0(\partial P)$. For $\x$ in $Y\setminus R$ set $w(\x)=S_P(S_{\partial P}^{-1}( u_{{|_{\partial P}}}))(\x)$. Since $w$ belongs to $W_3$ the difference $w-u$ also belongs to $W_3$.  Noting further that the traces of $w$ and $u$ agree on $\partial P$  we conclude that $w-u$ also belongs to $W_1\oplus W_2$. However  since $W_3=(W_1\oplus W_2)^\perp$ it is evident that $w-u=0$  and we conclude that $S_P^{-1}=S_{\partial P}^{-1}\gamma$.
 \end{proof}

The outward pointing normal derivative of a quantity $q$ on $\partial P$ is denoted by $\frac{\partial q}{\partial{n_\x}}$ and the  jump relations satisfied by the single layer potential are given by
\begin{eqnarray}
S_P(\phi)_{{|_{\partial P}^+}}=S_P(\phi)_{{|_{\partial P}^-}}, \hbox{ on $\partial P$},
\label{cont}
\frac{\partial S_P(\phi)}{\partial{n_\x}}{{|_{{\partial P}\stackrel{+}{-}}}}=\pm\frac{1}{2}\phi+K_P^*(\phi), \hbox{ on $\partial P$}
\label{jump}
\end{eqnarray}
where the double layer $K_P^*$ is a bounded linear operator mapping $L_0^2(\partial P)$ into $L_0^2(\partial P)$ defined by
\begin{eqnarray}
K^*_P(\phi)=\int_{\partial P} \frac{\partial G(\x,\y)}{\partial n_\x}\phi(\y)dS_\y,
\label{double}
\end{eqnarray}
and the subscripts $+$ and $-$ indicate traces from the outside and inside of $P$ respectively. Here $\frac{\partial G}{\partial n_\x}$ is a continuous kernel of order zero and it follows \cite{Folland,ColtonKress} that $K^*_P$ is a compact operator mapping $L_0^2(\partial P)$ into $L_0^2(\partial P)$.
Now we identify the transform $T$ defined by \eqref{resonanceT}.

\begin{theorem}
\label{T} 

The linear map $T:W_3\rightarrow W_3$ defined by the sesquilinear form \eqref{resonanceT} is given by
\begin{eqnarray}
T=S_P K_P^*S_P^{-1},
\label{ident}
\end{eqnarray}
and is a compact bounded  self-adjoint operator on $W_3$ with eigenvalues lying inside $(-1/2,1/2)$. 
\end{theorem}
\noindent Since $T$ is self-adjoint, bounded, and compact: the spectrum is discrete with only zero as an accumulation point, the eigenspaces associated with distinct eignevalues are pairwise orthogonal, finite dimensional, and their union together with the null space of $T$ is $W_3$.
Theorem \ref{completeeigen} follows immediately from Theorem \ref{T} noting that the choices $\lambda=1/2$, $\lambda=-1/2$ in \eqref{resonance} correspond to $\phi$ belonging to $W_1$ and $W_2$ respectively.

We now prove Theorem \ref{T}.
\begin{proof}
\noindent Step I.

 We establish \eqref{ident}. For $u$, $v$, consider
\begin{eqnarray}
(S_P K_p^*S_P^{-1}(u),v)=\int_{Y\setminus R}\nabla S_PK_p^*S_P^{-1}(u)\cdot\nabla \overline{v}\,dx.
\label{startup}
\end{eqnarray}
Integration by parts gives
\begin{eqnarray}
(S_PK_p^*S_P^{-1}(u),v)=\int_{\partial P}\left[\frac{\partial(S_PK_p^*S_P^{-1}(u))}{\partial n_\x}\right]_+^-\overline{v}\,dS,
\label{startup1}
\end{eqnarray}
where $[q]_+^-=q|_{{\partial P}^-}-q|_{{\partial P}^+}$. Applying the jump condition \eqref{jump} we get
\begin{eqnarray}
(S_PK_p^*S_P^{-1}(u),v)=-\int_{\partial P}K_p^*S_P^{-1}(u)\overline{v}\,dS,
\label{startup2}
\end{eqnarray}
applying \eqref{jump} again gives 
\begin{eqnarray}
K_p^*S_P^{-1}(u)&=&\frac{1}{2}\frac{\partial S_PS_P^{-1}(u)}{\partial n_\x}|_{{\partial P}^-}+\frac{1}{2}\frac{\partial S_PS_P^{-1}(u)}{\partial n_\x}|_{{\partial P}^+}\nonumber\\
&=&\frac{1}{2}\frac{\partial u}{\partial n_\x}|_{{\partial P}^-}+\frac{1}{2}\frac{\partial u}{\partial n_\x}|_{{\partial P}^+}.
\label{startup3}
\end{eqnarray}
Step I now follows on substitution of \eqref{startup3} into \eqref{startup2} and integration by parts.

\noindent Step II.
The remaining properties of the operator $T$ follow directly from the representation $T=S_PK_P^*S_P^{-1}$ and the properties of $S_P$ and $K_P^*$.
\end{proof}

We conclude noting that the spectrum of $T$ is the same as the spectrum of $K_P^*$. This is stated in the following Lemma.
\begin{lemma}
\label{samespectra}
$Tu=\lambda u$ if and only if $\lambda$ corresponds to an eigenvalue of $K_P^*$.
\end{lemma}
\begin{proof}
If a pair $(\lambda,u)$ belonging to $(-1/2,1/2)\times W_3$ satisfies $Tu=\lambda u$ then  $S_PK_P^*S_P^{-1}u=\lambda u$. Multiplication of both sides by $S_P^{-1}$ shows that $S_P^{-1}u$ is an eigenfunction for $K_P^*$ associated with $\lambda$.
Suppose the pair $(\lambda,w)$ belongs to $(-1/2,1/2)\times L_0^2(\partial P)$ and satisfies $K_P^*w=\lambda w$. Since the trace map from $W_3$ to $H^{1/2}(\partial P)\cap L_0^2(\partial P)$ is onto then there is a $u$ in $W_3$ for which $w=S_P^{-1}u$ and $K_P^*S_P^{-1}u=\lambda S_P^{-1}u$. Multiplication of this identity by $S_P$ shows that $u$ is an eigenfunction for $T$ associated with $\lambda$.
\end{proof}

\section{Existence for exterior problems with a dielectric permitivity in $\mathbb{C}$}
\label{exist}

In this section we establish Theorem \ref{existencetheoremforoutsideR}.
Recall that  any element $u$ of $H_{per}^1(Y\setminus R)/\mathbb{C}$ can be written as
\begin{eqnarray}
u=\mathcal{P}_1 u + \mathcal{P}_2 u +\sum_{-\frac{1}{2}<\lambda_n<\frac{1}{2}} \mathcal{P}_{\lambda_n} u +D.
\label{expandHper}
\end{eqnarray}
where $D$ is chosen such that $\int_{Y\setminus R} u \, dy=0$. Here $\mathcal{P}_i$ are the orthgonal projections onto $W_i$, $i=1,2$ and $\mathcal{P}_{\lambda_n}$ are the orthogonal projections  associated with  $\lambda_n$ in $(-1/2,1/2)$.
The orthogonal decomposition \eqref{expandHper} is used in the proof of Theorem \ref{existencetheoremforoutsideR} given below.

\begin{proof}

For  $u$, $v$ in $H_{per}^1(Y\setminus R)/\mathbb{C}$ we apply \eqref{expandHper} to see that
{\footnotesize
\begin{eqnarray}
 B_z(u,v) &=&\sum_{i=1}^2 B_z(\mathcal{P}_i u,\mathcal{P}_i v)+\sum_{-\frac{1}{2}<\lambda_n<\frac{1}{2}} B_z(\mathcal{P}_{\lambda_n} u, \mathcal{P}_{\lambda_n} v)\nonumber\\
 &=&(\mathcal{P}_1 u,v)+z(\mathcal{P}_2 u,v)+\sum_{-\frac{1}{2}<\lambda_n<\frac{1}{2}} (1+(z-1)(\frac{1}{2}-\lambda_n))(\mathcal{P}_{\lambda_n} u,v)\nonumber\\
 &=&(T_z u,v).
\label{bilinearform}
\end{eqnarray}}
From \eqref{bilinearform} we conclude that
\begin{eqnarray}
T_z =\mathcal{P}_1 +z \mathcal{P}_2 +\sum_{-\frac{1}{2}<\lambda_n<\frac{1}{2}}  (1+(z-1)(\frac{1}{2}-\lambda_n))\mathcal{P}_{\lambda_n}.
\label{Tz2}
\end{eqnarray}
It is evident from \eqref{Tz2} that for $z\not=(\lambda_i+1/2)/(\lambda_i-1/2)$ that $T_z$ is a bounded one to one and onto map in $H_{per}^1(Y\setminus R)/\mathbb{C}$. The formula for $T_z^{-1}$ is given by
\begin{eqnarray}
T_z^{-1}=\mathcal{P}_1 +z^{-1} \mathcal{P}_2 +\sum_{-\frac{1}{2}<\lambda_n<\frac{1}{2}}  (1+(z-1)(\frac{1}{2}-\lambda_n))^{-1}\mathcal{P}_{\lambda_n}.
\label{tzinverse}
\end{eqnarray}
Taking conjugates on both sides of \eqref{bilinearform} gives $\overline{B_z(u,v)}=(v,T_z u)$ and choosing $u=T_z^{-1}q$ for $q$ in $H_{per}^1(Y\setminus R)/\mathbb{C}$ delivers the identity

\begin{eqnarray}
\label{conjidenty}
\overline{B_z(T_z^{-1}q,v)}=(v,q).
\end{eqnarray}

To complete the proof consider any linear functional  $F$ in $[H_{per}^1(Y\setminus R)/\mathbb{C}]^*$ with $F(v)=0$ for $v=const.$
Applying the Reisz representation theorem shows that there exists a unique solution $u$ in $H_{per}^1(Y\setminus R)/\mathbb{C}$ of 
\begin{eqnarray}
B_z(u,v)=\overline{F(v)}, \hbox{ for all $v$ in $H_{per}^1(Y\setminus R)/\mathbb{C}$},
\label{finishproof}
\end{eqnarray}
\end{proof}
\noindent and Theorem \ref{existencetheoremforoutsideR} is proved.

\section{Convergence of the Power Series}
\label{convergence}

In this section we show that the series \eqref{xipowerseries}, \eqref{convgncee}  identified in Theorems \ref{convergenceofxi} and \ref{summable} are convergent. Here the convergence radius $R$ depends upon the branch of the quasistatic dispersion relation $\xi_0^n=\xi_0^n(\tau,\hat{\kappa})$.   The methodology applied here uses generating functions and follows the approach presented in \cite{FLS2}.

\noindent{Step I. {\em A priori estimates.}}

We establish a priori bounds for the solutions of \eqref{Rproblem} and \eqref{varformpsimprimeYminusR}. Before proceeding recall the Friderich's inequality  and trace estimate satisfied by elements $\psi$ of
$H^1_{per}(Y\setminus R)/\mathbb{C}$ given by 
\begin{eqnarray}
\Vert\psi\Vert_{Y\setminus R}=\left(\int_{Y\setminus R} |\nabla \psi|^2+|\psi|^2\,dy\right)^{1/2}\leq \Omega\Vert\psi\Vert
\label{Fredrichs}
\end{eqnarray}
and 
\begin{eqnarray}
\Vert \psi\Vert_{H^{1/2}(\partial R)}\leq A \Vert\psi\Vert
\label{trace}
\end{eqnarray}
where $\Omega$ and $A$ are positive constants depending on $R$ and $\Vert\psi\Vert$ is the norm on $H_{per}^1(Y\setminus R)/\mathbb{C}$ defined by \eqref{innerproduct}.

The solutions $\tilde{\psi}$ of \eqref{Rproblem} solve boundary value problems of the following generic form. Given $\psi\in H^1_{per}(Y\setminus R)$, $G  \in L^2(R) $ and $F\in [L^2(R)]^2$  the function $\tilde{\psi}$ is the $H^1(R)$ solution of 
\begin{equation}
 \begin{cases}
  \int_R (\nabla\tilde{\psi}\cdot\nabla \overline v-\xi_0 \psi\overline v)=\int_R G\overline v+\int_R F\cdot\nabla \overline v ~~~~~~ \forall v\in H_0^1(R)\\
\tilde{\psi}|_{\partial R-}=\psi|_{\partial R+} .
 \end{cases}
\end{equation}
Decompose $\tilde{\psi}$ according to $\tilde{\psi}=\psi_0+\psi_1+\psi_2$ such that
{\footnotesize
\begin{equation}
 \begin{cases}
  \int_R (\nabla\psi_0\cdot\nabla \overline v-\xi_0 \psi_0\overline v)=\int_R \left [G+(1+\xi_0)\psi_1
+(1+\xi_0)\psi_2\right]\overline v ~~\forall v\in H_0^1(R)\\
\\
\psi_0|_{\partial R-}=0,
 \end{cases}
\label{psi_0}
\end{equation}}
\begin{equation}
 \begin{cases}
  \int_R (\nabla\psi_1\cdot\nabla \overline v +\psi_1\overline v)=0 ~~~\forall v\in H_0^1(R)\\
\\  
\psi_1|_{\partial R-}=\psi|_{\partial R+},
 \end{cases}
\label{psi_1}
\end{equation}
and
\begin{equation}
 \begin{cases}
  \int_R (\nabla\psi_2\cdot\nabla \overline v + \psi_2\overline v)=\int_R F\cdot\nabla \overline v ~~~~~~ \forall v\in H_0^1(R)\\
\\
\psi_2|_{\partial R-}=0.
 \end{cases}
\label{psi_2}
\end{equation}
Let $J=G+(1+\xi_0)\psi_1+(1+\xi_0)\psi_2$ denote the  right hand side for (\ref{psi_0}).
Expanding $\psi_0$ and $J$ with respect to the complete orthonormal set of Dirichlet eigenfunctions $\{\phi_j\}_{j=1}^\infty$, $\{\nu_j\}_{j=1}^\infty$ gives 
\begin{eqnarray}
\Vert \psi_0\Vert_{H^1(R)}^2=\sum_{j=1}^\infty(1+\nu_j)\left|\frac{\langle J \psi_j\rangle_R}{(\xi_0-\nu_j)}\right|^2\leq C_{1,n}^2\Vert J\Vert_{L^2(R)}^2,
\label{expr}
\end{eqnarray}
where 
\begin{eqnarray}
C_{1,n}=\max_{\xi_0\in I_n}\left\{max_j \left\{\frac{(1+\nu_j)^{1/2}}{|\xi_0-\nu_j|}\right\}\right\}.
\label{intest}
\end{eqnarray}
Collecting results gives
\begin{eqnarray}
 ||\psi_1||_{H^1(R)}&=& ||\psi_2||_{H^{1/2}(\partial R)}\leq A ||\psi||\nonumber
\\||\psi_2||_{H^1(R)}&\leq &||F||_{L^2(R)} \nonumber
\\||\psi_0||_{H^1(R)}&\leq & C_{1,n}||G+(1+\xi)\psi_1+(1+\xi)\psi_2||_{L^2(R)}
\end{eqnarray}
and
\begin{eqnarray}
 ||\tilde{\psi}||_{H^1(R)}\leq B_n( ||G||_{L^2(R)}+||F||_{L^2(R)}+||\psi||),
\label{boundnessforpsiinR}
\end{eqnarray}
where 
\begin{eqnarray}
B_n=\max_{\xi_0\in I_n}\{ C_{1,n} ,C_{1,n} (1+\xi_0), A(C_{1,n}(1+\xi_0)+1)\}.
\label{Kn}
\end{eqnarray}

Now we estimate the field $\psi_*$. The explicit expression (\ref{psi*}) for $\psi_*$ gives
\begin{eqnarray}
 ||\psi_*||^2_{H^1(R)}=\sum_{n=1}^\infty\frac{(1+\mu_n)\mu_n^2<\phi_n>^2_R}{(\mu_n-\xi_0)^4},
 \label{psi*norm}
\end{eqnarray}
and
\begin{eqnarray}
 ||\psi_*||_{H^1(R)}\leq  L_n.
\label{boundnessforpsistarinR}
\end{eqnarray}
where  $L_{n}$ is the maximum of the right hand side of \eqref{psi*norm} for $\xi_0(\tau,\hat\kappa )\in  I_n$.

Next we provide an upper bound for the solutions of $\psi_m'$ of \eqref{varformpsimprimeYminusR}.  We may continuously extend elements $v$ in $H_{per}^1(Y\setminus R)/\mathbb{C}$ onto $R$ as $H_{per}^1(Y)$ functions such that the extension $v^{ext}$ satisfies
\begin{eqnarray}
\Vert v^{ext}\Vert_Y\leq C\Vert v\Vert.
\label{extchion}
\end{eqnarray}
In what follows we continue to denote these extensions by $v$ and \eqref{varformpsimprimeYminusR} takes the equivalent form \eqref{varformpsimprime}. Application of H\"olders inequality to the right hand side of \eqref{varformpsimprime} gives
{\footnotesize
\begin{eqnarray}
&&\tau^2|B_z(\psi_m',v)|\nonumber
\\&&\leq C\left(\Vert F_1\Vert_{L^2(Y\setminus R)} +\Vert G_1\Vert_{L^2(Y\setminus R)} +\Vert F_2\Vert_{L^2(R)} +\Vert G_2\Vert_{L^2(R)}\right)\Vert v\Vert.
\label{upest1}
\end{eqnarray}}
\par
It is evident that  $B_z(\psi_m',v)=(T_z \psi_m',v)$  and choosing $v=T_{\overline{z}}^{-1}\psi_m'$ gives $B_z(\psi_m',v)=\Vert \psi_m'\Vert^2, \Vert v\Vert\leq G_z\Vert \psi_m'\Vert$ and delivers the estimate
{\footnotesize
\begin{eqnarray}
\tau^2\Vert\psi_m'\Vert^2&=&
\tau^2|B_z(\psi_m',v)|\nonumber\\
 &\leq& C\times G_z(\Vert F_1\Vert_{L^2(Y\setminus R)} +\Vert G_1\Vert_{L^2(Y\setminus R)} \label{Blowerestimate}
 \\&&+\Vert F_2\Vert_{L^2(R)} +\Vert G_2\Vert_{L^2(R)})\Vert \psi_m'\Vert\nonumber
\end{eqnarray}}
with
\begin{eqnarray}
G_z=\max\left\{1,|z|^{-1},\sup_{-\frac{1}{2}<\lambda_n<\frac{1}{2}}\{|1+(z-1)(1/2-\lambda_n)|\}^{-1}\right\}.
\label{constzlowB}
\end{eqnarray}
For $z=\epsilon_p(\xi_0)$ we maximize \eqref{constzlowB} for $\xi_0\in I_n$  to obtain 
{\footnotesize
\begin{eqnarray}
 \tau^2\Vert\psi_m'\Vert\leq G_n\left(\Vert F_1\Vert_{L^2(Y\setminus R)} +\Vert G_1\Vert_{L^2(Y\setminus R)} +\Vert F_2\Vert_{L^2(R)} +\Vert G_2\Vert_{L^2(R)}\right).
 \label{upperprimeest1}
\end{eqnarray}}
Applying  Freidrich's inequality we arrive at the desired upper bound given by
{\footnotesize
\begin{eqnarray}
 \tau^2\Vert\psi_m'\Vert_{Y\setminus R}\leq E_n\left(\Vert F_1\Vert_{L^2(Y\setminus R)} +\Vert G_1\Vert_{L^2(Y\setminus R)} +\Vert F_2\Vert_{L^2(R)} +\Vert G_2\Vert_{L^2(R)}\right),
 \label{boundnessforpsiooutsideR}
\end{eqnarray}}
where $E_n$ denotes a generic constant depending only on $I_n$.

The a priori estimate for $\hat{\psi}$ follows from the representation formula \eqref{psihatformula}
and
\begin{eqnarray}
\Vert \hat{\psi}\Vert_{Y\setminus R}\leq H_n,
\label{hatpsiest}
\end{eqnarray}
where the constant $H_n$ depends only on $I_n$.

\noindent{Step II. \em{System of inequalities.}}

We apply the a priori estimates developed in Step I to the system of equations \eqref{Rproblem}, \eqref{varformpsimprimeYminusR}, and solvability conditions \eqref{iterationforxim1} derived in section \ref{higherorder}.
From \eqref{varformpsimprime} and (\ref{boundnessforpsiooutsideR}), it follows that 
{\footnotesize
\begin{eqnarray}
 \lefteqn{\tau^2\Vert\psi_m'\Vert_{Y\setminus R}\leq}\nonumber
\\ &&E_{n}\big[\tau^2(\sum_{l=1}^{m-2}|\xi_l|\Vert\psi_{m-l}\Vert_{Y\setminus R}+\sum_{l=1}^{m-2}|\xi_l|\Vert\psi_{m-1-l}\Vert_{Y\setminus R}+\sum_{l=1}^{m-2}|\xi_l|\Vert\psi_{m-2-l}\Vert_{Y\setminus R}\nonumber
\\&&+\Vert\psi_{m-1}\Vert_{Y\setminus R}+\Vert\psi_{m-2}\Vert_{Y\setminus R})+\sum_{l=0}^{m-2}|\xi_l|\Vert\psi_{m-2-l}\Vert_{H^1(R)}+\sum_{l=0}^{m-3}|\xi_l|\Vert\psi_{m-3-l}\Vert_{H^1(R)}\nonumber
\\&&+\sum_{l=1}^{m-4}|\xi_l|\Vert\psi_{m-4-l}\Vert_{H^1(R)}+\Vert\psi_{m-2}\Vert_{H^1(R)}+\Vert\psi_{m-3}\Vert_{H^1(R)}+\Vert\psi_{m-4}\Vert_{H^1(R)} \nonumber
\\&&+\sum_{l=0}^{m-2}\sum_{n=0}^l|\xi_{m-2-l}||\xi_n|(\Vert\psi_{l-n}\Vert_{Y\setminus R}+\Vert\psi_{l-n}\Vert_{H^1(R)})\nonumber
\\&&+\sum_{l=0}^{m-2}|\xi_l|(\Vert\psi_{m-2-1}\Vert_{Y\setminus R}+\Vert\psi_{m-2-1}\Vert_{H^1(R)})\big].
\label{bound1}
\end{eqnarray}}
From \eqref{Rproblem}  and (\ref{boundnessforpsiinR}), it follows that 
{\footnotesize
\begin{eqnarray}
\lefteqn{||\tilde\psi_{m}||_{H^1(R)}\leq }\nonumber
\\&&{B_n}\big[\sum_{l=1}^{m-1}\xi_l||\psi_{m-l}||_{H^1(R)}+\sum_{l=0}^{m-1}\xi_l||\psi_{m-1-l}||_{H^1(R)}\nonumber
\\&&+\sum_{l=0}^{m-2}\xi_l||\psi_{m-2-l}||_{H^1(R)}+||\psi_{m-1}||_{H^1(R)}+||\psi_{m-2}||_{H^1(R)}\nonumber
\\&&+\sum_{l=1}^{m-1}\sum_{n=0}^l\xi_{m-l}\xi_n||\psi_{l-n}||_{H^1(R)}+\sum_{l=1}^{m-1}\xi_l||\psi_{m-l}||_{H^1(R)}+||\psi_{m}||_{Y\setminus R}\big ]
\label{bound2}
\end{eqnarray}}

The solvability constant $\mathcal{G}$ in \eqref{iterationforxim1} is bounded away from zero and $\infty$ for $\xi_0\in I_n$  
Put $m-2\mapsto m$ in \eqref{iterationforxim1} to obtain the inequality for $|\xi_m|$. The inequality is in terms of a constant $C_n$ depending only on $I_n$ and is given by
{\footnotesize
\begin{eqnarray}
\lefteqn{|\xi_m|\leq}\nonumber
\\&& C_n \big [\tau^2\big(\sum_{l=0}^{m-1}|\xi_l|\Vert\psi_{m+1-l}\Vert_{Y\setminus R}+\sum_{l=0}^{m-1}|\xi_l|\Vert\psi_{m-l}\Vert_{Y\setminus R}+2\Vert\psi_{m+1}'\Vert_{Y\setminus R}+\Vert\psi_{m}\Vert_{Y\setminus R}\big)\nonumber
\\&&+\sum_{l=0}^{m-1}|\xi_l|\Vert\psi_{m-1-l}\Vert_{H^1(R)}+\sum_{l=0}^{m-2}|\xi_l|\Vert\psi_{m-2-l}\Vert_{H^1(R)}\nonumber
\\&&+\sum_{l=1}^{m-1}\sum_{n=0}^l|\xi_{m-l}||\xi_n|(\Vert\psi_{l-n}\Vert_{H^1(R)}+\Vert\psi_{l-n}\Vert_{Y\setminus R})\label{bound3}
\\&&+2\sum_{l=1}^{m-1}|\xi_l|(\Vert\psi_{m-l}\Vert_{H^1(R)}+\Vert\psi_{m-l}\Vert_{Y\setminus R})\nonumber
\\&&+(\Vert\tilde\psi_{m}\Vert_{H^1(R)}+\Vert\psi_{m}\Vert_{Y\setminus R})+(\Vert\tilde\psi_{m}\Vert_{H^1(R)}+\Vert\psi_{m}\Vert_{Y\setminus R})
\big]\nonumber
\end{eqnarray}}
The decompositions (\ref{decomp}), \eqref{decompp} and bounds (\ref{boundnessforpsistarinR}), \eqref{hatpsiest} give
\begin{eqnarray}
 \Vert\psi_m\Vert_{H^1(R)}&\leq & \Vert\tilde\psi_{m}\Vert_{H^1(R)}+L_n |\xi_m| ~~\hbox{ and}
\label{bound4}\\
\Vert \psi_m\Vert_{Y\setminus R}&\leq& \Vert\psi_m'\Vert_{Y\setminus R}+ H_n|\xi_{m-1}|.
\label{bound55}
\end{eqnarray}
\par
Here the upper bounds \eqref{bound1}, \eqref{bound2} \eqref{bound3}, \eqref{bound4}, and \eqref{bound55} hold for all values of $\tau$ for which the branch of the quasistatic dispersion relation $\xi_n^0$ lies in interval $I_n$.

\noindent{Step III. {\em Majorizing sequence and analyticity of generating functions.}

We simplify the exposition by introducing
\begin{eqnarray}
\overline p_m &=& ||\tau^m\psi_m||_{Y\setminus R}\\
p_m &=& ||\tau^m\psi_m||_{H^1(R)}\\
\tilde p_m &=& ||\tau^m\tilde\psi_m||_{H^1(R)}\\
p_m' &=& ||\tau^m\psi_m'||_{Y\setminus R}\\
s_m &=& |\tau^m\xi_m|,
\end{eqnarray}
and show that the corresponding series $\sum \rho^m\overline p_m$, $\sum \rho^m p_m$, $\sum \rho^m \tilde p_m$, $\sum \rho^m p_m'$, and $\sum \rho^m s_m$ 
converge. This is sufficient to establish summability and convergence for the series \eqref{convgncee} and \eqref{xipowerseries}. 
Writing \eqref{bound1}, \eqref{bound2} \eqref{bound3}, \eqref{bound4}, and \eqref{bound55} in terms of the new notation gives the following system of inequalities.
For $m\geq 2$ we have:
{\footnotesize
\begin{eqnarray}
\lefteqn{ p_{m}'\leq} \nonumber
\\&& E_n \big[ \sum_{l=1}^{m-2}s_l\overline p_{m-l}+\tau\sum_{l=0}^{m-2}s_l\overline p_{m-1-l}+ +\tau^2\sum_{l=0}^{m-2}s_l\overline p_{m-2-l}\nonumber
\\&&+\tau\overline p_{m-1}+\tau^2\overline p_{m-2}+\sum_{l=0}^{m-2}s_l p_{m-2-l}+\tau\sum_{l=0}^{m-3}s_l p_{m-3-l}+\tau^2\sum_{l=0}^{m-4}s_lp_{m-4-l}\nonumber
\\&&+p_{m-2}+\tau p_{m-3}+\tau^2 p_{m-4}+\sum_{l=0}^{m-2}\sum_{n=0}^ls_{m-2-l}s_n(\overline p_{l-n}+p_{l-n})\nonumber
\\&&+\sum_{l=0}^{m-2}s_l(\overline p_{m-2-l}+p_{m-2-l})
\big],
\end{eqnarray}}
and for $m\geq 1$,
{\footnotesize
\begin{eqnarray}
 \lefteqn{\tilde p_m\leq }\nonumber
\\&&B_n\big[\sum_{l=1}^{m-1}s_lp_{m-l}+\tau\sum_{l=0}^{m-1}s_lp_{m-1-l}+\tau^2\sum_{l=0}^{m-2}s_lp_{m-2-l}+\tau p_{m-1}\nonumber
\\&&+\tau^2 p_{m-2}+\sum_{l=1}^{m-1}\sum_{n=0}^ls_{m-l}s_n p_{l-n}+\sum_{l=1}^{m-1}s_lp_{m-l}+\overline p_m
\big].
\end{eqnarray}}
The bounds (\ref{bound3}), (\ref{bound4}), and \eqref{bound55} yield the following bounds for $m\geq 1$:
{\footnotesize
\begin{eqnarray}
 \lefteqn{s_m\leq }\nonumber
\\&&C_n \big[ 2\tau p_{m+1}'+\tau \sum_{l=1}^{m-1}s_l\overline p_{m+1-l}+\tau^2\sum_{l=0}^{m-1}s_l\overline p_{m-l}\nonumber
\\&&+\tau \sum_{l=0}^{m-1}s_l p_{m-1-l}+\tau^2\sum_{l=0}^{m-2}s_l p_{m-2-l}+\sum_{l=0}^{m-3}\sum_{n=0}^ls_{m-l}s_n(\overline p_{l-n}+p_{l-n})\nonumber
\\&&+2\sum_{l=1}^{m}s_l(\overline p_{m-l}+p_{m-l})+2\tilde p_m+ (2+\tau^2)\overline p_{m}\big],
\end{eqnarray}}
\begin{eqnarray}
p_m&\leq& \tilde p_m+L_n s_m,\label{pmbd}\\
\overline{p}_m&\leq& p_m'+\tau H_n s_{m-1}.\label{pmoverlinebd}
\end{eqnarray}
The  inequalities presented above hold for all $\tau$ for which $\xi_0^n$ belongs to $I_n$. In what follows it is convenient to increase the upper bounds by choosing the maximum value of $\tau$ for $\xi_0^n$ in $I_n$ and absorbing this value into the constants $B_n,C_n,E_n$  
Additionally any constant factors other than unity that multiply terms in these upper bounds are absorbed into these constants.  With these choices the system is written:
{\footnotesize
\begin{eqnarray}
\lefteqn{ p_{m}'\leq} \label{pprimebd}
\\&& E_n \big[ \sum_{l=1}^{m-2}s_l\overline p_{m-l}+\sum_{l=0}^{m-2}s_l\overline p_{m-1-l}+ +\sum_{l=0}^{m-2}s_l\overline p_{m-2-l}\nonumber
\\&&+\overline p_{m-1}+\overline p_{m-2}+\sum_{l=0}^{m-2}s_l p_{m-2-l}+\sum_{l=0}^{m-3}s_l p_{m-3-l}+\sum_{l=0}^{m-4}s_lp_{m-4-l}\nonumber
\\&&+p_{m-2}+ p_{m-3}+ p_{m-4}+\sum_{l=0}^{m-2}\sum_{n=0}^ls_{m-2-l}s_n(\overline p_{l-n}+p_{l-n})\nonumber
\\&&+\sum_{l=0}^{m-2}s_l(\overline p_{m-2-l}+p_{m-2-l})
\big],\nonumber
\end{eqnarray}}
{\footnotesize
\begin{eqnarray}
 \lefteqn{\tilde p_m\leq }\label{tpmbd}
\\&&B_n\big[\sum_{l=1}^{m-1}s_lp_{m-l}+\sum_{l=0}^{m-1}s_lp_{m-1-l}+\sum_{l=0}^{m-2}s_lp_{m-2-l}+ p_{m-1}+ p_{m-2}\nonumber
\\&&+\sum_{l=1}^{m-1}\sum_{n=0}^ls_{m-l}s_n p_{l-n}+\sum_{l=1}^{m-1}s_lp_{m-l}+\overline p_m
\big],\nonumber
\end{eqnarray}}
{\footnotesize
\begin{eqnarray}
 \lefteqn{s_m\leq }\label{smbd}
\\&&C_n \big[  p_{m+1}'+ \sum_{l=1}^{m-1}s_l\overline p_{m+1-l}+\sum_{l=0}^{m-1}s_l\overline p_{m-l}\nonumber
\\&&+ \sum_{l=0}^{m-1}s_l p_{m-1-l}+\sum_{l=0}^{m-2}s_l p_{m-2-l}+\sum_{l=0}^{m-3}\sum_{n=0}^ls_{m-l}s_n(\overline p_{l-n}+p_{l-n})\nonumber
\\&&+\sum_{l=1}^{m}s_l(\overline p_{m-l}+p_{m-l})+\tilde p_m+ \overline p_{m}\big],\nonumber
\end{eqnarray}}
\begin{eqnarray}
p_m&\leq& C_4(\tilde p_m+ s_m),\label{pmbd2}\\
\overline{p}_m&\leq& C_5(p_m'+ s_{m-1}),\label{pmoverlinebd2}
\end{eqnarray}
where $C_4$, $C_5$ are positive constants depending only on $I_n$ and $m\geq 1$.

We now introduce a majorizing sequence $\{(a_m,b_m,c_m,d_m,e_m)\}_{m=0}^\infty$ for which
\begin{eqnarray}
\sum_{n=0}^\infty a_n\rho^n=&\rho a(\rho)+\overline{p}_0&\geq\sum_{n=0}^\infty \overline{p}_n\rho^n\nonumber\\
\sum_{n=1}^\infty b_n\rho^n=&b(\rho)&\geq\sum_{n=1}^\infty \tilde{p}_n\rho^n\nonumber\\
\sum_{n=0}^\infty c_n\rho^n=&c(\rho)&\geq\sum_{n=0}^\infty s_n\rho^n\nonumber\\
\sum_{n=0}^\infty d_n\rho^n=&d(\rho)&\geq\sum_{n=0}^\infty p_n\rho^n\nonumber\\
\sum_{n=2}^\infty e_n\rho^n=&\rho e(\rho)&\geq\sum_{n=2}^\infty p_n'\rho^n, \label{majorant}
\end{eqnarray}
to show that the generating functions $a(\rho),\ldots,e(\rho)$ are analytic in a neighborhood of the origin $\rho=0$.

The majorizing sequence is chosen so that the system of inequalities \eqref{pprimebd}, \eqref{tpmbd}, \eqref{smbd}, \eqref{pmbd2}, \eqref{pmoverlinebd2}
hold with equality for $a_m=\overline{p}_m$, $b_m=\tilde{p}_m$, $c_m=s_m$, $p_m=d_m$, and $e_m=p_m'$.
Indeed, for this choice one observes that $\overline p_m\leq  a_m$, $\tilde p_m\leq  b_m$ , $s_m\leq  c_m$, $p_m\leq d_m$, and $p_m'\leq e_n$. 
Enforcing equality in  \eqref{pprimebd}, \eqref{tpmbd}, \eqref{smbd}, \eqref{pmbd2}, \eqref{pmoverlinebd2}, multiplying each by the appropriate power of $\rho$ and summation delivers the equivalent  system for the generating functions $a(\rho), b(\rho), c(\rho), d(\rho), e(\rho)$ given by:
\begin{eqnarray}
A_i(a,b,c,d,e,\rho)=0, \hbox{  $i=1,\ldots 5$},
\end{eqnarray}
where:
\begin{eqnarray}
A_1(a,b,c,d,e,\rho)&=&\overline{p}_1-a+C_5(e+(c-s_0)),\label{A1}
\end{eqnarray}
\begin{eqnarray}
 \lefteqn{A_2(a,b,c,d,e,\rho)=}\label{A2}
\\&& -b+B_n\big[(c-s_0)(d-p_0)+(cd+d)(\rho+\rho^2)+c^2 d-s_0^2 p_0-\rho s_0 s_1 (p_0+p_1)\nonumber
\\&&-s_0 p_0(c-s_0)-s_0 c d+s_0(\rho(s_0 p_1+s_1 p_0)+s_0 p_0)+(c-s_0)(d-p_0)+\rho a
\big]\nonumber
\end{eqnarray}
{\footnotesize
\begin{eqnarray}
 \lefteqn{A_3(a,b,c,d,e,\rho)=}\label{A3}
\\&&-(c-s_0)+C_n\big[ (a-\overline{p}_1)(c-s_0)+e+\rho a c+\rho(a+cd)\nonumber
\\&& +\rho^2c d+c^2(\rho a+\overline{p}_0)-s_0 \overline{p}_0 c-s_0((\rho a+\overline{p}_0)c-s_0\overline{p}_0)+(c-s_0)(d-p_0)+b\nonumber
\\&&-\rho(\rho a+\overline{p}_0)c s_1-\rho^2s_2(\rho a+\overline{p}_0)c+c^2d-s_0p_0 c-s_0(c d-s_0 p_0)-\rho cd (s_1+\rho s_2)\nonumber
\big]
\end{eqnarray}}
\begin{eqnarray}
 A_4(a,b,c,d,e,\rho)= (d-p_0)+C_4(b+(c-s_0)).\label{A4}
\end{eqnarray}
\begin{eqnarray}
 \lefteqn{A_5(a,b,c,d,e,\rho)=}\label{A5}
\\&&-e+E_n\big[ (a-\overline{p}_1)(c-s_0)+\rho ac+\rho(\rho a +\overline{p}_0)c +\rho a+\rho(\rho a+\overline{p}_0)\nonumber
\\&& +(c+1)d(\rho+\rho^2+\rho^3)+\rho c(c+1)(\rho a+\overline{p}_0+d)\nonumber
\big]
\end{eqnarray}
One can check $A_1,A_2,A_3,A_4$ and $A_5$ vanish for $a=\overline{p}_1=\Vert \psi_1\Vert_{Y\setminus R}$, $b=0$, $c=s_0=|\xi_0^n|$, $d=p_0=\Vert \psi_0\Vert_{H^1(R)}$, $e=0$, and $\rho=0$. Moreover, the determinant of the Jacobian matrix with respect to the first five variables at 
this point, i.e., $(a,b,c,d,e,\rho)=(\overline{p}_1,0,s_0,p_0,0,0)$, is 
\begin{eqnarray}
 det\frac{\partial (A_1,A_2,A_3,A_4,A_5)}{\partial(a,b,c,d,e)}(\overline{p}_1,0,s_0,p_0,0,0)=1.
\end{eqnarray}
Thus the implicit function theorem for functions of several complex variables guarantees that the generating functions $a,b,c,d,e$ are analytic within a neighborhood of $\rho=0$ and we conclude that the series on the right hand side of the inequalities \eqref{majorant}  are convergent.
Here $\overline{p}_0=\Vert \psi_0\Vert_{Y\setminus R}=|Y\setminus R|^{1/2}$ and the initial data $p_0$, $\overline{p}_1$ are bounded functions of $\xi_0^n$ on $I_n$. Choosing  the maximum values of  $p_0$, $\overline{p}_1$, $\xi_0^n$, as $\xi_0^n$ varies over $I_n$ delivers a majorizing sequence with radius of convergence $R$ depending only on $I_n$. It now follows that the series \eqref{xipowerseries} converges  and that \eqref{convgncee} is summable for $0\leq\rho<R$.


\section{Power Series Solution of Cell Problem }
\label{powerseriessolutionofcellproblem}
In this section we show that the functions defined by the power series solve \eqref{variational2}.  For $\xi_0^n$ in $I_n$ recall the series \eqref{upower}, \eqref{xipowerexpansion}
\begin{eqnarray*}
 &&u=\sum_{m=0}^\infty \eta^m u_m \label{upowerrho}
\\&&\xi=\sum_{m=0}^\infty\eta^m\xi_m \label{xipowerexpansionrho}
\end{eqnarray*}
with  $u_m=i^m\underline{u}_0\psi_m$ and $\eta=\rho\tau$ converge for $0<\rho<R$. 
For $v\in H^1_{per}(Y) $ and $0\leq \rho< R$. We define
{\footnotesize
\begin{eqnarray}
&&a^\eta(v):=\int_H\tau^2(\xi-\epsilon_r\frac{\omega_p^2}{c^2})(\nabla_y+i\eta \hat{\kappa})u\cdot \overline{(\nabla_y+i\eta \hat{\kappa})v }\nonumber
\\&&+\int_P\tau^2\xi(\nabla_y+i\eta \hat{\kappa})u\cdot \overline{(\nabla_y+i\eta \hat{\kappa})v }\nonumber
\\&&+\int_R\eta^2(\xi-\epsilon_r\frac{\omega_p^2}{c^2})(\nabla_y+i\eta \hat{\kappa})u\cdot \overline{(\nabla_y+i\eta \hat{\kappa})v }-\int_Y\eta^2\xi(\xi-\epsilon_r\frac{\omega_p^2}{c^2})u\overline{v} .
\end{eqnarray}}
First observe that $a^\eta(v)$ has a convergent power series in $\eta$ that is obtained by inserting $u$ and $\xi$ into the above expression and expanding in powers of $\eta$ . The coefficients of this expansion are exactly the left-hand side of equation (\ref{summation2}). Since  $\psi_m$ satisfies (\ref{summation2}), all coefficients of the power series expansion of $a^\eta(v)$ are zero and we conclude that $a^\eta(v)=0$ for all $v\in H^1_{per}(Y)$, $0\leq \rho< R$.   Thus we conclude that the pair $(u,\xi)$ satisfies (\ref{variational2}). 

Next we record the following property of the sequence $\{\xi_m\}$ that follows directly from the variational formulation of the problem and the power series solution.

\begin{theorem}
The functions  $\xi_m$ are real for all $m$.
\label{realvalue}
\end{theorem}
\begin{proof}
Setting $v=u$ in (\ref{variational2}) delivers a quadratic equation for $\xi$. Since the discriminant greater than zero, we conclude that $\xi$ is real and it follows that $\xi_m$, $m=1,\ldots$, are real. 
\end{proof}
\noindent With these results in hand, Theorems \ref{convergenceofxi} and \ref{summable} now follow from the convergence established in section \ref{convergence} together with Theorem \ref {Existencesequence}.

\section{Acknowledgments}
This research is supported by NSF grant DMS-0807265, DMS-1211066, AFOSR grant FA9550-05-0008, AFOSR MURI Grant FA95510-12-1-0489 administered through the University of New Mexico, and NSF EPSCOR Cooperative Agreement No. EPS-1003897 with additional support from the Louisiana Board of Regents.

\end{document}